\documentclass[12pt,a4paper]{article}
\usepackage[english]{babel}
\usepackage{amssymb,amsthm,amsmath}
\usepackage{hyperref}
\usepackage[textsize=small,textwidth=1.2in]{todonotes}

\newtheorem{theorem}{Theorem}[section]
\newtheorem{proposition}[theorem]{Proposition}
\newtheorem{corollary}[theorem]{Corollary}
\newtheorem{lemma}[theorem]{Lemma}

\theoremstyle{definition}
\newtheorem{definition}[theorem]{Definition}
\newtheorem{example}[theorem]{Example}

\theoremstyle{remark}
\newtheorem{remark}[theorem]{Remark}

\newcommand{\wt}{\mathrm{wt}}
\newcommand{\rk}{\mathrm{rk}}
\newcommand{\cl}{\mathrm{cl}}
\newcommand{\Rsupp}{\mathrm{Rsupp}}
\newcommand{\qbinom}[2]{\genfrac{[}{]}{0pt}{}{#1}{#2}}

\title{Defining the $q$-analogue of a matroid}
\author{R. Jurrius \and R. Pellikaan}

\begin{document}

\maketitle

\begin{abstract}
This paper defines the $q$-analogue of a matroid and establishes several properties like duality, restriction and contraction. We discuss possible ways to define a $q$-matroid, and why they are (not) cryptomorphic. Also, we explain the motivation for studying $q$-matroids by showing that a rank metric code gives a $q$-matroid. 
\end{abstract}

This paper establishes the definition and several basic properties of $q$-matroids. Also, we explain the motivation for studying $q$-matroids by showing that a rank metric code gives a $q$-matroid. We give definitions of a $q$-matroid in terms of its rank function and independent spaces. The dual, restriction and contraction of a $q$-matroid are defined, as well as truncation, closure, and circuits. Several definitions and results are straightforward translations of facts for ordinary matroids, but some notions are more subtle. We illustrate the theory by some running examples and conclude with a discussion on further research directions involving $q$-matroids. \\

Many theorems in this article have a proof that is a straightforward $q$-analogue of the proof for the case of ordinary matroids. Although this makes them appear very easy, we feel it is needed to include them for completeness and also because it is not a guarantee that $q$-analogues of proofs exist.

\section{$q$-Analogues}

The $q$-analogue of the number $n$ is defined by
\[ [n]_q = 1 + q + \cdots + q^{n-1} = \frac{q^n-1}{q-1}. \]
This forms the basis of \emph{quantum calculus}, and we refer to Kac and Cheung \cite{kac:2002} for an introduction to the subject. In combinatorics, one can view the $q$-analogue as what happens if we generalize from a finite set to a finite dimensional vector space. The ``$q$'' in $q$-analogue does not only refer to quantum, but also to the size of a finite field. In the latter case, $[n]_q$ is the number of $1$-dimensional vector spaces of a vector space $\mathbb{F}_q^n$; but also in general, we can view $1$-dimensional subspaces of a finite dimensional space as the $q$-analogues of the elements of a finite set. In this text we keep in mind finite fields, because of applications, but we will consider finite dimensional vector spaces over both finite and infinite fields. \\

Most notions concerned with sets have a straightforward $q$-analogue, as given in the following table:

\begin{center}
\begin{tabular}{|c|c|}
\hline
finite set & finite dim space \\
\hline
element & $1$-dim subspace \\
$\emptyset$ & $\mathbf{0}$ \\
size & dimension \\
$n$ & $\frac{q^n-1}{q-1}$ \\
intersection & intersection \\
union & sum \\
\hline
\end{tabular}
\end{center}
Furthermore, the Newton binomial
\[ \binom{n}{k} = \frac{n!}{k! (n-k)!} = \frac{n\cdots (n-k+1)}{1\cdots k} \]
counts the number of subsets of $\{1, \ldots ,n\}$ of size $k$. The $q$-analogue is given by the Gaussian binomial, or $q$-binomial
\[ \qbinom{n}{k}_q = \frac{[n]_q!}{[k]_q! [n-k]_q!} = \frac{(q^n-1)\cdots (q^{n-k+1}-1)}{(q-1)\cdots (q^k-1)}. \]
If we consider $q$ as the size of a finite field, the $q$-binomial counts the number of subspaces of $\mathbb{F}_q^n$ of dimension $k$. For infinite fields, we get a polynomial in $q$ that can be considered as the counting polynomial of the Grassmann variety of $k$-dimensional subspaces of an $n$-dimensional vector space, see \cite{plesken:2009}.

In most cases, we can go from the $q$-analogue to the ``normal'' case by taking the limit for $q\to1$. This can also be viewed as projective geometry over the field $\mathbb{F}_1$, as is nicely explained by Cohn \cite{cohn:2004}. \\

Two notions that were not mentioned above, because they need a bit more caution, are the difference and the complement. When taking the difference $A-B$ of two subsets $A$ and $B$, we mean ``all elements that are in $A$ but not in $B$''. The $q$-analogue of this would be ``all $1$-dimensional subspaces that are in $A$ but not in $B$''. The problem is that when $A$ and $B$ are finite dimensional spaces, all these $1$-dimensional subspaces together do not form a subspace. Sometimes this is not a problem, as we will see for example in property (I3) later on. We have several options for $A-B$ as a subspace.  We can take a subspace $C$ with $C\cap B=\mathbf{0}$ and $C\oplus(A\cap B)=A$. However, this space is not uniquely defined. We can also take the orthogonal complement, but this has the disadvantage that $A\cap A^\perp$ can be non-trivial. Using the quotient space as a complement will lower the dimension of the ambient space, which makes it perfect for the definition of contraction but not very suitable for other purposes. \\
The solution to this problem is to use all options described above, depending on for which property of $A-B$ we need a $q$-analogue.

\section{Rank function}

Although it is not strictly necessary to know about matroids before defining their $q$-analogue, the subject probably makes a lot more sense with ordinary matroids in mind. A great resource on matroids is Oxley \cite{oxley:2011}. Another one, that we will follow for in our search for cryptomorphic definitions of a $q$-matroid and the proofs of their equivalence, is Gordon and McNulty \cite{gordon:2012}.

\begin{definition}
A $q$-matroid $M$ is a pair $(E,r)$ in which $E$ is a finite dimensional vector space over a field $\mathbb{F}$ and $r$ an integer-valued function defined on the subspaces of $E$, called the \emph{rank}, such that for all subspaces $A,B$ of $E$:
\begin{itemize}
\item[(r1)] $0\leq r(A)\leq\dim A$
\item[(r2)] If $A\subseteq B$, then $r(A)\leq r(B)$.
\item[(r3)] $r(A+B)+r(A\cap B)\leq r(A)+r(B)$
\end{itemize}
\end{definition}

Note that this definition is a straightforward $q$-analogue of the definition of a matroid in terms of its rank. In the same way, we define the following.

\begin{definition}
Let $M=(E,r)$ be a $q$-matroid and let $A$ be a subspace of $E$. If $r(A)=\dim A$, we call $A$ an \emph{independent space}. If not, $A$ is called \emph{dependent}. If $A$ is independent and $r(A)=r(E)$, we call $A$ a \emph{basis}. The \emph{rank of $M$} is denoted by $r(M)$ and is equal to $r(E)$. A $1$-dimensional subspace that is dependent, is called a \emph{loop}.
\end{definition}

These definitions might cause some confusion at first: we assign a rank to a subspace that has little to do with its dimension, and we call a complete subspace (in)dependent. However, we stick to these notions because they are a direct $q$-analogue of what happens in ordinary matroids. Before we go to an example, we prove a Lemma that will be used repeatedly.

\begin{lemma}\label{unit-rank-increase}
Let $(E,r)$ be a $q$-matroid. Let $A$ be a subspace of $E$ and let $x$ be a $1$-dimensional subspace of $E$. Then $r(A+x)\leq r(A)+1$.
\end{lemma}
\begin{proof}
First note that for any $q$-matroid $r(\mathbf{0})=0$ and $r(x)$ is either $0$ or $1$, by (r1). Now apply property (r3) to $A$ and $x$:
\begin{eqnarray*}
r(A+x) & = & r(A+x) + 0 \\
 & = & r(A+x)+r(A\cap x) \\
 & \leq & r(A)+r(x) \\
 & \leq & r(A)+1.
\end{eqnarray*}
\end{proof}

\begin{example}\label{ex-uniform}
Let $E$ be a finite dimensional vector space of dimension $n$. Let $0\leq k\leq n$ be an integer. Define a function $r$ on the subspaces of $E$ as follows:
\[ r(A)=\left\{ \begin{array}{ll}
\dim A & \text{if }\dim A\leq k \\
k & \text{if }\dim A>k
\end{array} \right. \]
To show that $(E,r)$ is a $q$-matroid, we have to show that $r$ satisfies the properties (r1),(r2),(r3). First of all, $r$ is an integer valued function. It is clear from the definition of $r$ that (r1) and (r2) hold. For (r3), let $A,B$ be subspaces of $E$. We distinguish three cases, depending on the dimensions of $A$ and $B$. \\
If $r(A)=\dim A$ and $r(B)=\dim B$, then the definition of $r$ implies that $r(A\cap B)=\dim A\cap B$. By the modularity of dimension and (r2) it follows that
\begin{eqnarray*}
r(A+B)+r(A\cap B) & = & r(A+B)+\dim A\cap B \\
 & = & r(A+B)+\dim A+\dim B-\dim(A+B) \\
 & = & r(A+B)+r(A)+r(B)-\dim(A+B) \\
 & \leq & r(A)+r(B).
\end{eqnarray*}
If $r(A)=r(B)=k$, this implies that also $r(A+B)=k$. Since $r(A\cap B)\leq k$ by definition, we have that
\begin{eqnarray*}
r(A+B)+r(A\cap B) & \leq & k+k \\
 & = & r(A)+r(B).
\end{eqnarray*}
Finally, let $r(A)=\dim A$ and $r(B)=k$. Since $\dim B\geq k$, we also have that $\dim A+B\geq k$, hence $r(A+B)=k$.
\begin{eqnarray*}
r(A+B)+r(A\cap B) & = & k+r(A\cap B) \\
 & \leq & k+\dim A\cap B \\
 & \leq & k+\dim A \\
 & \leq & r(B)+r(A).
\end{eqnarray*}
We conclude that $(E,r)$ is indeed a $q$-matroid. We call it the \emph{uniform $q$-matroid} and denote it by $U_{k,n}$. Its independent spaces are all subspaces of dimension at most $k$, and its bases are all subspaces of dimension $k$.
\end{example}

The following two Propositions can be viewed as a variation of (r3). We will use them in later proofs.

\begin{proposition}\label{p-rank1}
Let $r$ be the rank function of a $q$-matroid $(E,r)$ and let $A,B$ be subspaces of $E$. Suppose $r(A+x)=r(A)$ for all $1$-dimensional subspaces $x\subseteq B$, $x\not\subseteq A$. Then $r(A+B)=r(A)$.
\end{proposition}
\begin{proof}
We prove this by induction on $k=\dim B-\dim(A\cap B)$. Let $\{x_1,\ldots,x_k\}$ be $1$-dimensional subspaces of $E$ that are in $B$ but not in $A$ such that $A+x_1+\cdots+x_k=A+B$. So the $x_i$ are generated by linearly independent vectors. Note that $k$ is finite, since $\dim(A+B)$ is finite and $k\leq\dim(A+B)$. If $k=0$, then $B\subseteq A$ so clearly $r(A+B)=r(A)$. \\
Now assume that $r(A+x_1+\cdots+x_t)=r(A)$ for all $t<k$.
We have to show that $r(A+x_1+\cdots+x_k)=r(A)$. By (r2) we have that $r(A)\leq r(A+x_1+\cdots+x_k)$.
By (r3) we have that
\begin{multline*}
r((A+x_1+\cdots+x_{k-1})+(A+x_k))+r((A+x_1+\cdots+x_{k-1})\cap(A+x_k)) \\
\leq r(A+x_1+\cdots+x_{k-1})+r(A+x_k)
\end{multline*}
which is equal to
\[ r(A+x_1+\cdots+x_k)+r(A)\leq r(A)+r(A) \]
and thus $r(A+x_1+\cdots+x_k)\leq r(A)$. We conclude that equality holds, and since $A+x_1+\cdots+x_k=A+B$, this proves the statement.
\end{proof}

\begin{proposition}\label{p-rank2}
Let $r$ be the rank function of a $q$-matroid $(E,r)$, let $A$ be a subspace of $E$ and let $x,y$ be $1$-dimensional subspaces of $E$. Suppose $r(A+x)=r(A+y)=r(A)$. Then $r(A+x+y)=r(A)$.
\end{proposition}
\begin{proof}
Applying (r3) to $A+x$ and $A+y$ gives the following equivalent statements:
\begin{eqnarray*}
r((A+x)+(A+y))+r((A+x)\cap(A+y)) & \leq & r(A+x)+r(A+y) \\
r(A+x+y)+r(A) & \leq & r(A)+r(A) \\
r(A+x+y) & \leq & r(A).
\end{eqnarray*}
On the other hand, by (r2) we have that $r(A)\leq r(A+x+y)$, so equality must hold.
\end{proof}

We end this section with a remark about the difference between matroids and $q$-matroids. Let $(\mathbb{F}_q^n,r)$ be a $q$-matroid defined over a finite field. Let $X$ be the set of $1$-dimensional subspaces of $E$ and define a function on the subsets of $X$ as follows:
\[ \rho(A)=r(\langle A\rangle), \]
that is, we take the rank in the $q$-matroid of the span of $A$. Then it is not difficult to show that $(X,\rho)$ is a matroid. However, this matroid behaves a lot different from the $q$-matroid that we started with. For example, it has $\frac{q^n-1}{q-1}$ elements and rank $n$, which means its rank is very low in comparison to its cardinality. Also, if we take the usual duality, we do not get the dual $q$-matroid (that we define later) because the complement of a subspace in $X$ is not a subspace. Similar remarks hold for restriction and contraction, as well as for the link with rank metric codes. In short, by changing to the matroid $(X,\rho)$, we lose a lot of the structure of the $q$-matroid $(\mathbb{F}_q^n,r)$.

\section{Independent spaces}

Now that we have defined a $q$-matroid in terms of its rank function, a logical question is to ask if we could also define it in terms of its independent spaces, bases, etcetera. Unfortunately, the answer to this question is not as easy as just taking the $q$-analogues of cryptomorphic definitions of an ordinary matroid. The goal of this section is to establish the next cryptomorphic definition of a $q$-matroid.

\begin{theorem}\label{indep-rank}
Let $E$ be a finite dimensional space. If $\mathcal{I}$ is a family of subspaces of $E$ that satisfies the conditions:
\begin{itemize}
\item[(I1)] $\mathcal{I}\neq\emptyset$.
\item[(I2)] If $J\in\mathcal{I}$ and $I\subseteq J$, then $I\in\mathcal{I}$.
\item[(I3)] If $I,J\in\mathcal{I}$ with $\dim I<\dim J$, then there is some $1$-dimensional subspace $x\subseteq J$, $x\not\subseteq I$ with $I+x\in\mathcal{I}$.
\item[(I4)] Let $A,B\subseteq E$ and let $I,J$ be maximal independent subspaces of $A$ and $B$, respectively. Then there is a maximal independent subspace of $A+B$ that is contained in $I+J$.
\end{itemize}
and  $r$ is the function defined by $r_{\mathcal{I}}(A)=\max\{\dim I:I\in\mathcal{I},I\subseteq A\}$  for all $A\subseteq E$, 
then $(E,r_{\mathcal{I}})$ is a $q$-matroid and  its family of independent spaces is equal to $\mathcal{I}$.\\
Conversely, if $\mathcal{I}_r$ is the family of independent spaces of a $q$-matroid $(E,r)$, 
then $\mathcal{I}_r$ satisfies the conditions (I1),(I2),(I3),(I4) and $r=r_{\mathcal{I}_r}$.
\end{theorem}

The first three properties are a direct $q$-analogue of the axioms we use when we define an ordinary matroid in terms of its independent sets. The property (I4) however is really needed, as the next example and counter example show.

\begin{example}\label{ex-2}
Let $E=\mathbb{F}_2^4$ and let $\mathcal{I}$ be the set of all subspaces of dimension at most $2$ that do not contain the $1$-dimensional space $\langle0001\rangle$. Now $\mathcal{I}$ is not empty, so it satisfies (I1). If a space does not contain $\langle0001\rangle$, then all its subspaces also do not contain $\langle0001\rangle$, hence (I2) holds. For (I3), the interesting case is to check for $\dim I=1$ and $\dim J=2$, with $I\not\subseteq J$. From all the three $1$-dimensional spaces $x$ in $J$, there can only be one such that $I+x$ contains $\langle0001\rangle$, hence we have proved (I3). We will see in the next section that $\mathcal{I}$ is indeed the family of independent subspaces of a $q$-matroid.
\end{example}

\begin{example}\label{counterexample}
Let $E=\mathbb{F}_2^4$ and let $\mathcal{I}$ be the family consisting of
\[ I=\left\langle\begin{array}{cccc} 1 & 0 & 0 & 1 \\ 0 & 1 & 1 & 0 \end{array}\right\rangle \]
and all its subspaces. It is not difficult to see that $\mathcal{I}$ satisfies (I1),(I2),(I3): in fact, $\mathcal{I}$ is the family of independent spaces of the uniform $q$-matroid $U_{2,2}$ embedded into the space $E$. Consider the subspaces
\[ A=\left\langle\begin{array}{cccc} 1 & 0 & 0 & 0 \\ 0 & 1 & 0 & 0 \\ 0 & 0 & 1 & 0 \end{array}\right\rangle, \quad B=\left\langle\begin{array}{cccc} 0 & 1 & 0 & 0 \\ 0 & 0 & 1 & 0 \\ 0 & 0 & 0 & 1 \end{array}\right\rangle. \]
Both $A$ and $B$ have $\langle0110\rangle$ as a maximal independent subspace. But $A+B=E$ has $I$ as a maximal independent subspace, and $I$ is not contained in $\langle0110\rangle$. So $\mathcal{I}$ does not satisfy (I4). \\
Let $r_\mathcal{I}$ be the rank function defined in Theorem \ref{indep-rank}. Then
\[ r_\mathcal{I}(A+B)+r_\mathcal{I}(A\cap B)=2+1>1+1=r_\mathcal{I}(A)+r_\mathcal{I}(B), \]
so property (r3) does not hold for $r_\mathcal{I}$ and $\mathcal{I}$ is not the family of independent spaces of a $q$-matroid.
\end{example}

In order to understand why we need the extra axiom (I4), let us investigate a bit what goes wrong in the counter example.

\begin{lemma}\label{loopsum}
Let $x$ and $y$ be loops of a $q$-matroid. Then the space $x+y$ has rank $0$.
\end{lemma}
\begin{proof}
Apply property (r3) to $x$ and $y$:
\begin{eqnarray*}
r(x+y) & = & r(x+y)+0 \\
 & = & r(x+y)+r(x\cap y) \\
 & \leq & r(x)+r(y) \\
 & = & 0.
\end{eqnarray*}
By (r1), it follows that $r(x+y)=0$.
\end{proof}

Or in other words: loops come in subspaces. This Lemma might look trivial, but it is exactly what goes wrong in Example \ref{counterexample}. Take the loops $\langle1000\rangle$ and $\langle0001\rangle$: their sum has rank $1$. The difference with ordinary matroids is that for sets, $A\cup B$ contains only elements that were already in either $A$ or $B$. In the $q$-analogue this is not true: the space $A+B$ contains $1$-dimensional subspaces that are in neither $A$ nor $B$. Therefore, it is ``more difficult'' to bound $r(A+B)$, making it also more difficult for property (r3) to hold.

\begin{remark}\label{r-axioms}
Let $\mathcal{I}$ be the family of independent spaces of a $q$-matroid with ground space $E$. Embed $\mathcal{I}$ in a space $E^\prime$ with $\dim E^\prime>\dim E$, resulting in a family $\mathcal{I}^\prime$. Then $\mathcal{I}^\prime$ is \emph{not} the family of independent spaces of a $q$-matroid over $E^\prime$. This is because all $1$-dimensional spaces that are in $E^\prime$ but not in $E$ are loops, but they do not form a subspace: this contradicts Lemma \ref{loopsum}. It follows that a set of axioms for $\mathcal{I}$ that is invariant under embedding can never be a full set of axioms that defines a $q$-matroid.
\end{remark}

Again, if we look back at Example \ref{counterexample}, we see that this counter example was created by embedding a uniform matroid in a space of bigger dimension. So in order to completely determine a $q$-matroid in terms of its independent spaces, we need an extra axiom that regulates how the spaces in $\mathcal{I}$ interact with the other subspaces of the $q$-matroid. This is what the axiom (I4) does. We will now prove in three steps that (I4) holds for every $q$-matroid.

\begin{proposition}\label{p-maxindep1}
Let $(E,r)$ be a $q$-matroid. Let $A\subseteq E$ and let $I$ be a maximal independent subspace of $A$. Let $x\subseteq E$ be a $1$-dimensional space. Then there is a maximal independent subspace of $A+x$ that is contained in $I+x$.
\end{proposition}
\begin{proof}
If $x\subseteq A$, the result is clear. If $r(A)=r(A+x)$ then $I$ is a maximal independent set in $A+x$ and $I\subseteq I+x$, so we are also done. Therefore assume that $x$ is not contained in $A$ and $r(A)\neq r(A+x)$. By Lemma \ref{unit-rank-increase} this means $r(A+x)=r(A)+1$. \\
If $A$ is independent, then $A+x=I+x$ also has to be independent, so the statement is proven. Assume that $A$ is not independent. Then there are $A^\prime,y\subseteq A$ such that $A=A^\prime+y$ and $I\subseteq A^\prime$, hence $r(A)=r(A^\prime)$.
Now we use Proposition \ref{p-rank2} on $A^\prime$, $x$, and $y$. We have that $r(A^\prime+x+y)=r(A+x)\neq r(A)$ by assumption, so $r(A^\prime)$, $r(A^\prime+x)$ and $r(A^\prime+y)$ can not all be equal since this would contradict Proposition \ref{p-rank2}. Because $r(A^\prime)=r(A)$ and $r(A^\prime+y)=r(A)$, it needs to be that $r(A^\prime+x)\neq r(A)$. In fact, $r(A^\prime+x)>r(A)$. \\
If $A^\prime$ is independent, then $A^\prime=I$ and we have that $A^\prime+x=I+x$ is independent as well. This proves the statement. If $A^\prime$ is not independent, we repeat the procedure above: find $A^{\prime\prime}$ and $y^\prime$ such that $A^\prime=A^{\prime\prime}+y^\prime$ and $I\subseteq A^{\prime\prime}$, and apply Proposition \ref{p-rank2}. We keep doing this until we arrive at $r(I+x)>r(A)$, which means $I+x$ is independent.
\end{proof}

This result has the following consequence. First of all, the result holds for all maximal independent subspaces $I\subseteq A$. Suppose that $r(A+x)=r(A)+1$. For all $1$-dimensional subspaces $z\subseteq A+x$, $z\not\subseteq A$, we have that $A+x=A+z$. Hence, for all these $z$, we have that $I+z$ is independent. Also, all these $z$ have to be independent themselves, by (I2). So if enlarging a space raises its rank, it means all added $1$-dimensional subspaces are independent and all combinations of $I+z$ have to be independent as wel.

\begin{proposition}\label{p-maxindep2}
Let $(E,r)$ be a $q$-matroid. Let $A\subseteq E$ and let $I$ be a maximal independent subspace of $A$. Let $B\subseteq E$. Then there is a maximal independent subspace of $A+B$ that is contained in $I+B$.
\end{proposition}
\begin{proof}
The proof goes by induction on $\dim B$. If $\dim B=0$ then the statement is trivially true. If $\dim B=1$ the statement is true by Proposition \ref{p-maxindep1} above. Assume $\dim B>1$ and the statement is true for all subspaces with dimension less then $\dim B$. \\
Let $B^\prime$ be a subspace of $B$ of codimension $1$. By the induction hypothesis there is a maximal independent subspace $J$ of $A+B^\prime$ that is contained in $I+B^\prime$. Let $x$ be a $1$-dimensional subspace $x\subseteq B$, $x\not\subseteq B^\prime$, so $B=B^\prime+x$ and $A+B=A+B^\prime+x$. Now apply Proposition \ref{p-maxindep1} to $A+B^\prime$ and $J$: there is a maximal independent subspace of $A+B$ that is contained in $J+x\subseteq I+B^\prime+x=I+B$.
\end{proof}

\begin{proposition}\label{p-maxindep3}
Let $(E,r)$ be a $q$-matroid. Let $A,B\subseteq E$ and let $I,J$ be maximal independent subspaces of $A$ and $B$, respectively. Then there is a maximal independent subspace of $A+B$ that is contained in $I+J$.
\end{proposition}
\begin{proof}
By Proposition \ref{p-maxindep2} there is a maximal independent subspace of $A+B$ that is contained in $I+B$. This subspace is also maximal independent in $I+B$, by (r2) and $I+B\subseteq A+B$. So $r(A+B)=r(I+B)$. On the other hand, if we apply the same Proposition \ref{p-maxindep2} to $B$ and $I$, we find a maximal independent subspace $K$ of $B+I$ that is contained in $J+I$. Again, $K$ is also maximal independent in $J+I$, by (r2) and $J+I\subseteq B+I$. So $r(I+B)=r(I+J)$. This implies that $r(A+B)=r(I+J)$, hence the subspace $K\subseteq I+J$ is maximal independent in $A+B$, as was to be shown.
\end{proof}

Before finally proving Theorem \ref{indep-rank}, we prove a variation of the properties (I1),(I2),(I3). We denote by $\mathbf{0}$ the $0$-dimensional subspace that contains only the zero vector.

\begin{proposition}\label{indep-prime}
Let $E$ be a finite dimensional space and let $\mathcal{I}$ be a family of subspaces of $E$.
Then the family $\mathcal{I}$ satisfies the properties (I1),(I2),(I3) above if and only if it satisfies:
\begin{itemize}
\item[(I1')] ${\bf 0} \in\mathcal{I}$.
\item[(I2)\phantom{'}] If $J\in\mathcal{I}$ and $I\subseteq J$, then $I\in\mathcal{I}$.
\item[(I3')] If $I,J\in\mathcal{I}$ with $\dim J=\dim I+1$, then there is some $1$-dimensional subspace $x\subseteq J$, $x\not\subseteq I$ with $I+x\in\mathcal{I}$.
\end{itemize}
\end{proposition}
\begin{proof}
We need to show that (I1),(I2),(I3)$\iff$(I1'),(I2),(I3'). \\
$\Rightarrow$: Since $\mathcal{I}\neq\emptyset$ by (I1) and every subspace of an independent space is independent by (I2), we have $\mathbf{0}\subseteq\mathcal{I}$ (I1'). (I3') is just a special case of (I3). \\
$\Leftarrow$: (I1') directly implies (I1). Let $I,J\in\mathcal{I}$ with $\dim I<\dim J$. Let $I^\prime$ be some subspace of $J$ with $\dim I^\prime=\dim I+1$. Then $I^\prime$ is independent by (I2) and we can use (I3') to find a $1$-dimensional subspace $x\subseteq I^\prime$, $x\not\subseteq I$ with $I+x\in\mathcal{I}$. Since $I^\prime\subseteq J$, clearly $x\subseteq J$, $x\not\subseteq I$, so (I3) follows.
\end{proof}

\begin{proof}[Proof (Theorem \ref{indep-rank})]
The proof consists of three parts.
\begin{enumerate}
\item $r\to\mathcal{I}$. Given a function $r$ with properties (r1),(r2),(r3), define $\mathcal{I}$ as $\{A\subseteq E:r(A)=\dim A\}$ and prove (I1),(I2),(I3),(I4).
\item $\mathcal{I}\to r$. Given a family $\mathcal{I}$ with properties (I1),(I2),(I3),(I4), define
$r(A)$ as $\max_{I\subseteq A}\{\dim I:I\in\mathcal{I}\}$ and prove (r1),(r2),(r3).
\item The first two are each others inverse, that is: $\mathcal{I}\to r\to\mathcal{I}^\prime$ implies $\mathcal{I}=\mathcal{I}^\prime$, and $r\to\mathcal{I}\to r^\prime$ implies $r=r^\prime$.
\end{enumerate}
$\bullet$ Part 1. Let $M=(E,r)$ be a $q$-matroid and define the family $\mathcal{I}$ to be those subspaces $I$ of $E$ for which $r(I)=\dim I$. We will show $\mathcal{I}$ satisfies (I1'),(I2),(I3'),(I4). \\
By (r1), $r(\mathbf{0})=0$, so $r(\mathbf{0})=\dim\mathbf{0}$ and $\mathbf{0}\in\mathcal{I}$, hence (I1'). (I4) was proven in Proposition \ref{p-maxindep3}. \\
For (I2), let $J\in\mathcal{I}$ and $I\subseteq J$. We use (r3) with $A=I$ and $B$ a subspace of $J$ such that $A\cap B= {\bf 0}$ and $A+B=J$, to show $\dim I=r(I)$. The following is independent of the choice of $B$. Since $\dim J=r(J)$, we have
\[ r(I+B)+r(I\cap B)=r(J)+r(\mathbf{0})=\dim J. \]
By (r1), we have
\[ r(I)+r(B)\leq\dim I+\dim(B)=\dim J. \]
Combining and using (r3) gives
\[ \dim J=r(J)+r(\mathbf{0})\leq r(I)+r(B)\leq\dim I+\dim(B)=\dim J, \]
so we must have equality everywhere. This means, with (r1), that $r(B)=\dim(B)$ and $r(I)=\dim I$.
Therefore $I\in\mathcal{I}$ and (I2) holds. \\
We will prove (I3') by contradiction. Let $I,J\in\mathcal{I}$ with $\dim I<\dim J$ and let $x$ a $1$-dimensional subspace $x\subseteq J$, $x\not\subseteq I$. Suppose that (I3) fails, so $I+x\notin\mathcal{I}$. Then we have $r(I)=\dim I$ but $r(I+x)\neq\dim(I+x)=\dim I+1$. By (r1) and (r2) we have that
\[ \dim I=r(I)\leq r(I+x)\leq\dim(I+x)=\dim I+1. \]
The second inequality can not be an equality, so the first inequality has to be an equality: $r(I+x)=r(I)$. Now this reasoning holds for every $1$-dimensional subspace $x\subseteq J$, $x\not\subseteq I$ so by Proposition \ref{p-rank1} we have that $r(I)=r(I+J)$. But $J\in\mathcal{I}$ and we have that
\[ r(I+J)=r(I)=\dim I<\dim J=r(J) \]
which contradicts (r2) because $J\subseteq I+J$. So (I3) has to hold. \\
$\bullet$ Part 2. Let $\mathcal{I}$ be a family of subspaces of $E$ that satisfies (I1),(I2),(I3),(I4). Define $r(A)$ to be the dimension of the largest independent space contained in $A$. We show $r$ satisfies (r1),(r2),(r3). \\
Since the rank is a dimension, it is a non-negative integer. From the definition of $r$ we have $r(A)\leq\dim A$ and from (I1') we have $0\leq r(A)$. This proves (r1). If $A\subseteq B\subseteq E$, then every independent subspace of $A$ is an independent subspace of $B$, so
\[ r(A)=\max_{I\subseteq A}\{\dim I:I\in\mathcal{I}\}\leq\max_{I\subseteq B}\{\dim I:I\in\mathcal{I}\}=r(B) \]
and thus (r2). The difficultly in this part is to prove (r3). \\
Let $A,B\subseteq E$ and let $I_{A\cap B}$ be a maximal independent space in $A\cap B$. Use (I3) as many times as possible to extend $I_{A\cap B}$ to a maximal independent space $I_A\subseteq A$, and the same to get a maximal independent space $I_B\subseteq B$. By Proposition \ref{p-maxindep3} there is a maximal independent space $I_{A+B}$ of $A+B$ that is contained in $I_A+I_B$. Furthermore, $I_A\cap I_B=I_{A\cap B}$ because $I_{A\cap B}\subseteq I_A$ and $I_{A\cap B}\subseteq I_B$, and $I_{A\cap B}$ is a maximal independent space in $A\cap B$ hence a maximal independent space in $I_A\cap I_B$. Combining all this, we have
\begin{eqnarray*}
r(A+B)+r(A\cap B) & = & \dim I_{A+B}+\dim I_{A\cap B} \\
 & \leq & \dim(I_A+I_B)+\dim I_{A\cap B} \\
 & = & \dim I_A+\dim I_B-\dim I_{A\cap B}+\dim I_{A\cap B} \\
 & = & \dim I_A+\dim I_B \\
 & = & r(A)+r(B)
\end{eqnarray*}
and this is exactly (r3). \\
$\bullet$ Part 3. Given a rank function $r$ satisfying (r1),(r2),(r3), create a family $\mathcal{I}$ by $I\in\mathcal{I}$ if $\dim I=r(I)$. Then use $\mathcal{I}$ to create a (possibly new) rank function $r^\prime(A)=\max_{I\subseteq A}\{\dim I:I\in\mathcal{I}\}$. We want to show that $r^\prime(A)=r(A)$ for all $A\subseteq E$. Note that $r^\prime(A)=\dim I=r(I)$ for some $I\subseteq A$. By (r2), $r(I)\leq r(A)$ so $r^\prime(A)\leq r(A)$. For the reverse inequality, assume $r^\prime(A)<r(A)$ for some $A\subseteq E$. Then by definition of $r^\prime$, for all $I\in\mathcal{I}$ with $I\subseteq A$ we must have $r(A)>\dim I$. Let $I$ be a maximum-dimension such space. Then for all $1$-dimensional subspaces $x\subseteq A$ that intersect trivially with $I$, we have $I+x\notin\mathcal{I}$. Thus $r(I)=r(I+x)$ for all such $x$ and by Proposition \ref{p-rank1} we have $r(I)=r(A)=\dim I$. Contradiction, so $r^\prime(A)\geq r(A)$. Together we have $r^\prime(A)=r(A)$. \\
Given a family $\mathcal{I}$ satisfying (I1),(I2),(I3),(I4), define $r$ by $r(A)=\max_{I\subseteq A}\{\dim I:I\in\mathcal{I}\}$. Then let $\mathcal{I}^\prime$ be defined by $I\in\mathcal{I}^\prime$ if $r(I)=\dim I$. We want to show that $\mathcal{I}=\mathcal{I}^\prime$. Let $I\in\mathcal{I}$, then $r(I)=\dim I$ by the definition of $r$, and thus $I\in\mathcal{I}^\prime$. Now let $I\in\mathcal{I}^\prime$, then $r(I)=\dim I$ by the definition of $\mathcal{I}^\prime$, and thus $I$ is the largest independent subspace of $I$ and $I\in\mathcal{I}$.
\end{proof}

\section{Rank metric codes}\label{sec-codes}

Now that we have established some basic facts about $q$-matroids, we are ready to discuss the motivation of studying them. We show that every rank metric code gives rise to a $q$-matroid. For more on rank metric codes, see Gabidulin \cite{gabidulin:1985}. We consider codes over $L$, where $L$ is a finite Galois field extension of a field $K$. This is a generalization of the case where $K=\mathbb{F}_q$ and $L=\mathbb{F}_{q^m}$ of Gabidulin's \cite{gabidulin:1985} to arbitrary characteristic as considered by Augot, Loidreau and Robert \cite{augot:2013,augot:2014}. Much of the material here about rank metric codes is taken from \cite{jurrius:2015a,jurrius:2016}. See also \cite{martinez-penas:2016}. \\

Let $K$ be a field and let $L$ be a finite Galois extension of $K$. A \emph{rank metric code} is an $L$-linear subspace of $L^n$. To all codewords we associate a matrix as follows. Choose a basis $B=\{ \alpha_1, \ldots ,\alpha _m \} $ of $L$ as a vector space over $K$. Let $\mathbf{c} =(c_1, \ldots ,c_n)\in L^n$. The $m \times n$ matrix $M_B(\mathbf{c})$ is associated to $\mathbf{c} $ where the $j$-th column of $M_B(\mathbf{c} )$ consists of the coordinates of $c_j$  with respect to the chosen basis: $c_j = \sum_{i=1}^m c_{ij}\alpha_i$. So $M_B(\mathbf{c} )$ has entries $c_{ij}$. \\
The $K$-linear row space in $K^n$ and the rank of $M_B(\mathbf{c} )$ do not depend on the choice of the basis $B$,
since for another basis $B'$ there exists an invertible matrix $A$ such that $M_B(\mathbf{c})= AM_{B'}(\mathbf{c})$.
If the choice of basis is not important, we will write $M(\mathbf{x})$ for $M_{B}(\mathbf{x})$.
The rank weight $\wt_R(\mathbf{c} )=\rk(\mathbf{c})$ of $\mathbf{c} $ is by definition the rank of the matrix $M(\mathbf{c} )$,
or equivalently the dimension over $K$ of the row space of $M_B(\mathbf{c} )$. This definition follows from the rank distance, that is defined by $d_R(\mathbf{x}, \mathbf{y} ) = \rk(\mathbf{x}-\mathbf{y})$. The rank distance is in fact a metric on the collection of all $m \times n$ matrices, see \cite{augot:2013,gabidulin:1985}.

\begin{definition}
Let $C$ be an $L$-linear code.
Let $\mathbf{c} \in C$. Then $\Rsupp(\mathbf{c})$, the \emph{rank support} of $\mathbf{c}$
is the $K$-linear row space of $M_B(\mathbf{c})$. So $\wt_R(\mathbf{c})$ is the dimension of $\Rsupp(\mathbf{c})$.
\end{definition}

Note that this definition is the rank metric case of the support weights, or weights of subcodes,
of codes over the Hamming metric.

\begin{definition}\label{dJ2}
For a $K$-linear subspace $J$ of $K^n$ we define:
\[ C(J)=\{\mathbf{c}\in C :\Rsupp(\mathbf{c})\subseteq J^\perp\}. \]
\end{definition}

From this definition it is clear that $C(J)$ is a $K$-linear subspace of $C$, but in fact it is also an $L$-linear subspace.

\begin{lemma}\label{lJ2}
Let $C$ be an $L$-linear code of length $n$ and let $J$ be a $K$-linear subspace of $K ^n$.
Then $\mathbf{c} \in C(J)$ if and only if $\mathbf{c}\cdot\mathbf{y}=0$ for all $ \mathbf{y} \in J $.
Furthermore $C(J)$ is an $L$-linear subspace of $C$.
\end{lemma}
\begin{proof}
The following statements are equivalent:
\[ \begin{array}{c}
\mathbf{c} \in C(J)\\
\sum_{j=1}^nc_{ij}y_j =0 \mbox{ for all } \mathbf{y} \in J \mbox{ and } i=1, \ldots ,m\\
\sum_{i=1}^m(\sum_{j=1}^nc_{ij}y_j) \alpha_i=0 \mbox{  for all } \mathbf{y} \in J\\
\sum_{j=1}^n(\sum_{i=1}^mc_{ij} \alpha_i) y_j=0 \mbox{ for all } \mathbf{y} \in J\\
\sum_{j=1}^n c_j y_j=0 \mbox{ for all }\mathbf{y} \in J\\
\mathbf{c} \cdot \mathbf{y} =0\mbox{ for all } \mathbf{y} \in J\\
\end{array} \]
Hence $C(J) = \{  \mathbf{c} \in C :  \mathbf{c} \cdot \mathbf{y}=0 \mbox{ for all } \mathbf{y} \in J  \} $.
From this description it follows directly that $C(J)$ is an $L$-linear subspace of $C$.
\end{proof}

\begin{definition}\label{dH3}
Let $C$ be an $L$-linear code of length $n$.
Let $J$ be a $K$-linear subspace of $K ^n$ of dimension $t$ with generator matrix $Y$.
Define the map $\pi _J : L^n \rightarrow L^t$ by $\pi _J (\mathbf{x})= \mathbf{x} Y^T$, and  $C_J = \pi_J(C)$.
\end{definition}

\begin{lemma}\label{lJ3}
Let $C$ be an $L$-linear code of length $n$.
Let $J$ be a $K$-linear subspace of $K ^n$ of dimension $t$ with generator matrix $Y$.
Then $\pi _J $ is an $L$-linear map and $C_J$ is an $L$-linear code of length $t$ and its dimension does not depend on the chosen
generator matrix. Furthermore we have an exact sequence of vector spaces:
\[ 0 \longrightarrow C(J)  \longrightarrow C  \longrightarrow C_J  \longrightarrow 0. \]
\end{lemma}
\begin{proof}
The map $\pi _J $ is defined by a matrix with entries in $K$ so it is $L $-linear.
The image of $C$ under $\pi_J$ is $C_J$. Hence $C_J$ is an $L$-linear code.\\
If $G$ is generator matrix of $C$, then  $C_J$ is the row space of $GY^T$ and the dimension of $C_J$ is equal to the rank of $GY^T$.
If $G'$ is another generator matrix of $C$ and $Y'$ another generator matrix of $J$, then
there exists an invertible $k\times k$ matrix $A$ with entries in $L$ and an invertible $t\times t$ matrix $B$ with entries in $K$
such that $G'=AG$ and $Y'=BY$. The row space of $G'(Y')^T$ is the space $C_J$ with respect $Y'$. Now
\[ G'(Y')^T=(AG)(BY)^T = A(GY^T) B^T, \]
and $A$ and $B^T$ are invertible. Hence $G'(Y')^T$ and $GY^T$ have the same rank. Therefore the dimension of $C_J$ does not depend on the chosen generator matrix for $J$.\\
The map $C(J)  \rightarrow C$ is injective and the map $\pi_J: C  \rightarrow C_J $ is surjective, both by definition.
Furthermore the kernel  of $\pi_J: C  \rightarrow C_J $ is equal to $\{  \mathbf{c} \in C :  \mathbf{c} \cdot \mathbf{y}=0 \mbox{ for all } \mathbf{y} \in J  \} $, which is equal to $C(J)$ by Lemma \ref{lJ2}.
Hence the given sequence is exact.
\end{proof}

\begin{definition}\label{dH4}
Let $C$ be an $L$-linear code of length $n$.
Let $J$ be a $K$-linear subspace of $K^n$ of dimension $t$.
Define $l(J) = \dim_{L} C(J)$ and $r(J) = \dim_{L} C_J$.
\end{definition}

\begin{corollary}\label{cJ4}
Let $C$ be an $L$-linear code of length $n$ and dimension $k$ and let $J$ be a $K$-linear subspace of $K^n$. Then $l(J)+r(J)=k$.
\end{corollary}
\begin{proof}
This is a direct consequence of Proposition \ref{lJ3}.
\end{proof}

We now know enough about rank metric codes to show that there is a $q$-matroid associated to them.

\begin{theorem}
Let $C$ be a linear rank metric code over $L$, $E=K^n$ and $r$ the function from Definition \ref{dH4}. Then $(E,r)$ is a $q$-matroid.
\end{theorem}
\begin{proof}
First of all, it is clear that $r$ is an integer valued function defined on the subspaces of $E$. We need to show that $r$ satisfies the properties (r1),(r2),(r3). Let $I,J\subseteq E$. We will make heavy use of Corollary \ref{cJ4}, saying $r(J)=k-l(J)$. \\
$\bullet$ (r1) $0\leq r(J)\leq\dim J$. \\
This follows from the definition of $r(J)= \dim C_J$ and the fact that $C_J$ is a subspace of $K^t$ with $t=\dim J$. \\
$\bullet$ (r2) If $I\subseteq J$ then $r(I)\leq r(J)$. \\
Let $I \subseteq J$ and let $ \mathbf{c}\in C(J)$. Then $I \subseteq J \subseteq \Rsupp(\mathbf{c})^\perp$. So $ \mathbf{c}\in C(I)$. Hence $C(J) \subseteq C(I)$ and $l(J) \leq l(I)$. Therefore $r(I) \leq r(J)$. \\
$\bullet$ (r3) $r(I+J) +r(I\cap J) \leq r(I) + r(J)$. \\
Let $I$, $J$ and $H$ be linear subspaces of $K^n$. If $I \subseteq H$  and $J  \subseteq H$, then $I +J \subseteq H$, since $H$ is a subspace. On the other hand, if $I+J \subseteq H$, then $I \subseteq I+J \subseteq H$ so $I \subseteq H$ and similarly $J \subseteq H$. Hence $I+J \subseteq H$  if and only if  $I \subseteq H$  and $J \subseteq H$.\\
The following statements are then equivalent:
\[ \begin{array}{c}
\mathbf{c}  \in C(I) \cap C(J)\\
\mathbf{c} \in C(I) \ \mbox{ and } \  \mathbf{c} \in C(J) \\
I \subseteq\Rsupp(\mathbf{c})^\perp \ \mbox{ and } \ J\subseteq\Rsupp(\mathbf{c})^\perp \\
I+J  \subseteq \Rsupp(\mathbf{c})^\perp \\
\mathbf{c} \in C(I+J)
\end{array} \]
Hence $C(I) \cap C(J) = C(I+J)$. \\
Now if $\mathbf{c} \in C(I)$ then $I \subseteq \Rsupp(\mathbf{c})^\perp$ so $I\cap J \subseteq (\Rsupp(\mathbf{c}))^\perp $. Hence $\mathbf{c} \in C(I\cap J)$. So $C(I) \subseteq C(I\cap J)$ and similarly $C(J) \subseteq C(I\cap J)$. Therefore $C(I)+C(J) \subseteq C(I\cap J)$.\\
Combining the above and using the modularity of dimension, we now have
\begin{eqnarray*}
l(I)+l(J) & = & \dim C(I) + \dim C(J) \\
 & = & \dim(C(I)\cap C(J))+\dim(C(I)+C(J)) \\
 & \leq & \dim(C(I+J))+\dim(C(I\cap J)) \\
 & = & l(I+J)+l(I\cap J)
\end{eqnarray*}
It follows that $r(I+J)+r(I\cap J) \leq r(I) + r(J)$. \\
We have shown that the function $r$ satisfies (r1),(r2),(r3), so we conclude that $(E,r)$ is indeed a $q$-matroid.
\end{proof}

\begin{corollary}
The rank of the $q$-matroid $M(C)$ associated to a rank metric code $C$ is $\dim C$.
\end{corollary}
\begin{proof}
We have that $r(M(C))=r(E)=\dim C-l(E)$ and also $E^\perp=\mathbf{0}$, so $C(E)=\mathbf{0}$ and $r(M(C))=\dim C$.
\end{proof}

\begin{corollary}
Let $L^\prime$ be a field extension of $L$ such that $L^\prime$ is Galois over $K$. Let $C\otimes L^\prime$ be the the $L^\prime$-linear code obtained by taking all $L^\prime$-linear combinations of words of $L$. Then the $q$-matroids associated to $C$ and $C\otimes L^\prime$ are the same.
\end{corollary}
\begin{proof}
We first show that $(C(I))\otimes L^\prime = (C\otimes L^\prime )(I)$. \\
Let $\mathbf{c} \in (C(I))\otimes L^\prime$.
Let $\mathbf{b}_1, \ldots , \mathbf{b}_l$ be a basis of $C(I)$ over $L$.
Then $\mathbf{b}_1, \ldots , \mathbf{b}_l$ is also a basis of $(C(I))\otimes L^\prime$ over $L^\prime$ by the definition of taking $\otimes L^\prime$.
Also, $\mathbf{b}_i \cdot \mathbf{x} =0$ for all $\mathbf{x} \in I$ by Lemma \ref{lJ2}.
There exist $\lambda _1, \ldots ,\lambda_l \in L^\prime$ such that
$\mathbf{c} = \sum_{i=1}^l \lambda_i \mathbf{b}_i$.
So by linearity $\mathbf{c} \cdot \mathbf{x} =0$ for all $\mathbf{x} \in I$,
hence $\mathbf{c}\in (C\otimes L^\prime )(I)$ by Lemma \ref{lJ2}.
Therefore  $(C(I))\otimes L^\prime \subseteq  (C\otimes L^\prime )(I)$.\\
Conversely, let $\mathbf{c} \in (C\otimes L^\prime  )(I)$.
Then $\mathbf{c} \cdot \mathbf{x} =0$ for all $\mathbf{x} \in I$.
Let $\mathbf{g}_1, \ldots , \mathbf{g}_k$ be a basis of $C$ over $L$.
Then $\mathbf{g}_1, \ldots , \mathbf{g}_k$ is also a basis of $C\otimes L^\prime$ over $L^\prime$.
There exist $\lambda _1, \ldots ,\lambda_k \in L^\prime$ such that
$\mathbf{c} = \sum_{i=1}^k \lambda_i\mathbf{g}_i$.
Let $\alpha_1, \ldots ,\alpha_m$ be a basis of $L^\prime $ over $L$.
Then for every $i$ there exist $\lambda _{i1}, \ldots ,\lambda_{im} \in L$ such that
$\lambda_i = \sum_{j=1}^m \lambda_{ij}\alpha_j$.
Let $\mathbf{x}\in I$. Then $\sum_{i=1}^k \lambda_{ij} \mathbf{g}_i \cdot \mathbf{x} \in L$ for all $j$,
$$
0 =\mathbf{c} \cdot \mathbf{x} =
\sum_{j=1}^m \left(\sum_{i=1}^k \lambda_{ij} \mathbf{g}_i \cdot \mathbf{x}\right) \alpha_j
$$
and $\alpha_1, \ldots ,\alpha_m$ is a basis of $L^\prime $ over $L$.
So $\sum_{i=1}^k \lambda_{ij} \mathbf{g}_i \cdot \mathbf{x}=0$ for all $j$ and all $\mathbf{x}\in I$.
Hence $\sum_{i=1}^k \lambda_{ij} \mathbf{g}_i \in C(I)$ for all $j$.
Therefore $\mathbf{c} \in (C(I))\otimes L^\prime$, and
$(C\otimes L^\prime )(I)  \subseteq (C(I))\otimes L^\prime$.\\
We conclude that $l(I)$, the dimension of $C(I)$ over $L$ is also the dimension of
$(C\otimes L^\prime  )(I)$ over $L^\prime$.
Hence the rank functions of the $q$-matroids $M(C)$ and $M((C\otimes L^\prime  ))$ are the same.
\end{proof}

\begin{example}\label{ex-code}
Let $L=\mathbb{F}_{8} $ and $K=\mathbb{F}_2 $. Let $a \in \mathbb{F}_{8}$ with $a^3=1+a$. Let $C$ be the rank metric code over $L$ with generator matrix
\[ G=\left(\begin{array}{cccc}
1 & a & 0 & 0 \\
0 & 1 & a & 0
\end{array}\right). \]
We can find the matroid associated to $C$ by finding its bases. They are independent, so their rank equals their dimension, which is $2$. These are the subspaces $J$ of $\mathbb{F}_2^4$ such that $l(J)=0$. This means $C(J)=\mathbf{0}$, i.e., there is no nonzero codeword such that $\Rsupp(\mathbf{c})\subseteq J^\perp$. Now $\wt_R(\mathbf{c})$ can not be $0$ unless $\mathbf{c}=\mathbf{0}$. It can only be $1$ if all nonzero entries in the codeword are the same: that can not happen. So if $\mathbf{c}$ is nonzero, the dimension of $\Rsupp(\mathbf{c})$ is at least $2$. On the other hand, all codewords have a zero in the last coordinate of their rank support. This means that if $J$ is perpendicular to $(0,0,0,1)$, there can not be a nonzero codeword that has $\Rsupp(\mathbf{c})\subseteq J^\perp$. The bases of $M(C)$ are thus the $2$-dimensional subspaces of $\mathbb{F}_2^4$ that do not contain $\langle0001\rangle$. This means the subspace $\langle0001\rangle$ is a loop. In fact, this is the matroid from Example \ref{ex-2}.
\end{example}

Using the theory of rank metric codes, we can learn more about the function $l(J)$.

\begin{definition}
Let $C$ be an $L$-linear code of length $n$. Then the \emph{dual} of $C$, notated by $C^\perp$, consists of all vectors of $L^n$ that are orthogonal to all codewords of $C$.
\end{definition}

The next Proposition is the $q$-analogue of the well-known fact that the minimum distance is the minimal number of dependent columns in a parity check matrix of the code.

\begin{proposition}\label{pJ}
Let $C$ be an $L$-linear code of length $n$. Then $t<d_R(C^\perp)$ if and only if $\dim_{L} (C_J)=t$ for all $K $-linear subspaces $J$ of $K^n$ of dimension $t$.
\end{proposition}
\begin{proof}
See \cite[Theorem 1]{gabidulin:1985}.
\end{proof}

\begin{lemma}\label{lJ5}
Let $C$ be an $L$-linear code of length $n$.
Let $d_R$ and $d_R^{\perp}$ be the minimum rank distance of $C$ and $C^{\perp}$, respectively.
Let $J$ be a $K$-linear subspace of $K ^n$ of dimension $t$. Let $l(J) = \dim_{L} C(J)$.
Then
$$
l(J)=\left\{\begin{array}{cl}
k-t & \text{for all } t<d^{\perp}_R \\
0 & \text{for all } t>n-d_R
\end{array}\right.
$$
\end{lemma}
\begin{proof}
The first inequality is a direct consequence of Proposition \ref{pJ}.\\
Let $t>n-d_R$ and let $\mathbf{c}\in C(J)$. Then $J$ is contained in the orthoplement of $\Rsupp (\mathbf{c})$, so $t\leq n-\wt_R(\mathbf{c})$.
It follows that $\wt_R(\mathbf{c})\leq n-t<d_R$, so $\mathbf{c}$ is the zero word and therefore $l(J)=0$.
\end{proof}

\begin{example}\label{ex-uniform-MRD}
Let $m\geq n$ and let $C$ be an $L$-linear code of length $n$, dimension $k$ and minimum distance $d_R=n-k+1$. Such a code is called an MRD (\emph{maximum rank distance}) code. Gabidulin \cite{gabidulin:1985} constructed such codes over finite fields for all $n$, $k$ and $q$. The construction was generalized to characteristic $0$ and rational function fields by Augot, Loidreau and Robert \cite{augot:2013,augot:2014}. The dual of an MRD code is again an MRD code and its minimum distance is therefore $d_R^\perp=k+1$. If we apply Lemma \ref{lJ5}, we find that the function $l(J)$ is completely determined in terms of the dimension $t$ of $J$:
\[ l(J)=\left\{\begin{array}{cl}
k-t & \text{for all } t\leq k \\
0 & \text{for all } t>k
\end{array}\right. \]
This means that also $r(J)$ is completely determined:
\[ r(J)=\left\{\begin{array}{cl}
t & \text{for all } t\leq k \\
k & \text{for all } t>k
\end{array}\right. \]
As we have seen in Example \ref{ex-uniform}, this is the rank function of the uniform $q$-matroid $U_{k,n}$.
\end{example}

\section{Truncation}

We present the notion of truncation of a $q$-matroid, so that we can use it in our proofs concerning axioms for bases. From now on we denote by $\mathcal{I}(M)$ the independent spaces of the $q$-matroid $M$, and if a $q$-matroid is defined by $\mathcal{I}$, we denote it by $(M,\mathcal{I})$.

\begin{definition}
Let $M=(E,\mathcal{I})$ be a $q$-matroid with $r(E)\geq 1$. The \emph{truncated matroid} $\tau(M)$ is a $q$-matroid with ground space $E$ and independent spaces those members of $\mathcal{I}$ that have dimension at most $r(M)-1$; so
\[ \mathcal{I}(\tau(M))=\{I\in\mathcal{I}:\dim I<r(M)\}. \]
\end{definition}

Because the dimension of an independent space is at most $r(M)$, this means that we simply remove all maximal independent spaces from $\mathcal{I}(M)$ to get $\mathcal{I}(\tau(M))$.

\begin{theorem}\label{t-trunc}
The truncation $\tau(M)$ is indeed a $q$-matroid, that is, $\mathcal{I}(\tau(M))$ satisfies (I1),(I2),(I3),(I4).
\end{theorem}
\begin{proof}
Because $r(M)\geq 1$, we have that $\mathbf{0}\in\mathcal{I}(\tau(M))$ hence (I1) holds. Let $J\in\mathcal{I}(\tau(M))$ and $I\subseteq J$. Then $\dim I\leq\dim J$ and $\dim J<r(M)$, so $\dim I<r(M)$ and $I\in\mathcal{I}(\tau(M))$. This proves (I2). For (I3) and (I4), it is enough to note that $\mathcal{I}(\tau(M))\subseteq\mathcal{I}(M)$. We conclude that $\tau(M)$ is indeed a $q$-matroid.
\end{proof}

We have the following straightforward description of the rank function of the truncated matroid:

\begin{corollary}
Let $M=(E,r)$ be a $q$-matroid with $r(M)\geq 1$. Then the rank function $r_\tau$ of the truncation $\tau(M)$ is give by
\[ r_\tau(A)=\min\{r(A),r(M)-1\}. \]
\end{corollary}

This means that for all subspaces $A$ of $E$ with $r(A)<r(M)$, we have $r(A)=r_\tau(A)$. Only for $r(A)=r(M)$ the rank gets down: $r_\tau(A)=r(A)-1$.

\begin{example}
Let $U_{k,n}$ be the uniform $q$-matroid of Example \ref{ex-uniform}. The truncation of $U_{k,n}$ has as independent spaces all subspaces of dimension at most $k-1$, so it is equal to $U_{k-1,n}$.
\end{example}

\begin{example}
Let $M$ the $q$-matroid of Example \ref{ex-2}. The truncation has as independent spaces $\mathbf{0}$ and all $1$-dimensional subspaces except $\langle0001\rangle$.
\end{example}

\section{Bases}

Remark \ref{r-axioms} about the axioms for independent spaces, holds for bases as well: if a set of axioms is invariant under embedding the family of bases $\mathcal{B}$ in a space of higher dimension, then it can not completely determine a $q$-matroid. This is why we need a fourth axiom.

\begin{theorem}\label{indep-bases}
Let $E$ be a finite dimensional space. If $\mathcal{B}$ is a family of subspaces of $E$ that satisfies the conditions:
\begin{itemize}
\item[(B1)] $\mathcal{B}\neq\emptyset$
\item[(B2)] If $B_1,B_2\in\mathcal{B}$ and $B_1\subseteq B_2$, then $B_1=B_2$.
\item[(B3)] If $B_1,B_2\in\mathcal{B}$, then for every codimension $1$ subspace $A$ of $B_1$ with $B_1\cap B_2\subseteq A$ there is a $1$-dimensional subspace $y$ of $B_2$ with $A+y\in\mathcal{B}$.
\item[(B4)] Let $A,B\subseteq E$ and let $I,J$ be maximal intersections of some bases with $A$ and $B$, respectively. Then there is a maximal intersection of a basis and $A+B$ that is contained in $I+J$.
\end{itemize}
and $\mathcal{I}$ is the family defined by $\mathcal{I}_\mathcal{B}=\{I:\exists B\in\mathcal{B},I\subseteq B \}$, then $(E,\mathcal{I}_\mathcal{B})$ is a $q$-matroid and its family of bases is $\mathcal{B}$. \\
Conversely, if $\mathcal{B}_\mathcal{I}$ is the family of bases of a $q$-matroid $(E,\mathcal{I})$, then $\mathcal{B}_\mathcal{I}$ satisfies the conditions (B1),(B2),(B3),(B4) and $\mathcal{I}=\mathcal{I}_{\mathcal{B}_\mathcal{I}}$.
\end{theorem}

\begin{remark}\label{r-ynotinB1}
In property (B3), it can not happen that $y\subseteq B_1$. Since by assumption $y\subseteq B_2$, this means that $y\subseteq B_1\cap B_2\subseteq A$. But then $A+y=A$, but $A$ can not be a basis because of (B2). So $y\notin B_1$.
\end{remark}

\begin{example}
Let $M$ be the $q$-matroid of Example \ref{ex-2}. The bases are the subspaces of dimension $2$ that do not contain $\langle0001\rangle$. We illustrate the property (B3). Let $B_1=\langle1100,0010\rangle$ and $B_2=\langle1010,0100\rangle$ . Then the intersection $B_1\cap B_2$ is $\langle1110\rangle$. This means we only have one choice for a codimension $1$ subspace of $B_1$ that contains $B_1\cap B_2$: it has to be $\langle1110\rangle$. If we add either $\langle1010\rangle$ or $\langle0100\rangle$ we get $B_2$, which is a basis.
\end{example}

Before proving the theorem, we first prove a slight variation of the axioms.
\begin{proposition}
Let $E$ be a finite dimensional space and let $\mathcal{B}$ be a family of subspaces of $E$.
Consider the condition:
\begin{itemize}
\item[(B2')] If $B_1,B_2\in\mathcal{B}$, then $\dim B_1=\dim B_2$.
\end{itemize}
The family $\mathcal{B}$ satisfies (B1),(B2),(B3) if and only if $\mathcal{B}$ satisfies (B1),(B2'),(B3).
\end{proposition}
\begin{proof}
$\Leftarrow$: It is clear that (B2') implies (B2). \\
$\Rightarrow$: Let $B_1,B_2\in\mathcal{B}$. Then use (B3) multiple times to form $B_1^\prime$, of the same dimension as $B_1$, that contains only $1$-dimensional subspaces that are also in $B_2$. So $B_1^\prime\subseteq B_2$ and by (B2) they must be equal. So $\dim B_1^\prime=\dim B_2$ and since $\dim B_1=\dim B_1^\prime$, we have $\dim B_1=\dim B_2$ and thus (B2').
\end{proof}

\begin{proof}[Proof (Theorem \ref{indep-bases})]
As with Theorem \ref{indep-rank}, the proof has three parts:
\begin{enumerate}
\item $\mathcal{I}\to\mathcal{B}$. Given $\mathcal{I}$ with properties (I1),(I2),(I3),(I4), define $\mathcal{B}$ as the family of independent spaces that are maximal with respect to inclusion and prove (B1),(B2),(B3),(B4).
\item $\mathcal{B}\to\mathcal{I}$. Given $\mathcal{B}$ with properties (B1),(B2),(B3),(B4), prove that the properties (I1),(I2),(I3),(I4) hold for $\mathcal{I}=\{I:\exists B\in\mathcal{B},I\subseteq B \}$.
\item The first two are each others inverse, that is:
$\mathcal{I}\to\mathcal{B}\to\mathcal{I}^\prime$ implies
$\mathcal{I}=\mathcal{I}^\prime$, and $\mathcal{B}\to\mathcal{I}\to\mathcal{B}^\prime$ implies $\mathcal{B}=\mathcal{B}^\prime$.
\end{enumerate}
$\bullet$ Part 1. Let $M=(E,\mathcal{I})$ be a $q$-matroid and define $\mathcal{B}$ to be the family of subspaces $B$ that are subspaces of $\mathcal{I}$ of maximal dimension, i.e.,
$\mathcal{B}=\{B\in\mathcal{I}:\forall B^\prime\in\mathcal{I}, B\subseteq B^\prime\Rightarrow B=B^\prime\}$.
We need to show that $\mathcal{B}$ satisfies (B1),(B2'),(B3),(B4). \\
Now (B1) is easy: since $\mathbf{0}\in\mathcal{I}$ by (I1') and $E$ is finite-dimensional, we can find an element $B\in\mathcal{I}$ which is not properly contained in any other independent space. Then $B\in\mathcal{B}$ and hence $\mathcal{B}\neq\emptyset$. \\
(B2') is also easy: if there are sets $B_1,B_2\in\mathcal{B}$ with $\dim B_1<\dim B_2$, then, since $B_1,B_2\in\mathcal{I}$ and (I3), there is a $1$-dimensional subspace $x\subseteq B_2$, $x\not\subseteq B_1$, so that $B_1+x\in\mathcal{I}$.
But $B_1$ is a proper subspace of $B_1+x$, which contradicts the definition of $\mathcal{B}$, so $\dim B_1=\dim B_2$ and hence (B2'). \\
Next we show (B3). Let $B_1,B_2\in\mathcal{B}$. Let $A$ be a codimension $1$ subspace of $B_1$ with $B_1\cap B_2=A\cap B_2$.
Since $B_1,B_2\in\mathcal{I}$ and $A$ is a subspace of $B_1$, we have $A\in\mathcal{I}$ by (I2).
Apply (I3) to $A$ and $B_2$: since $\dim(A)<\dim B_2$ by (B2'), there is a $1$-dimensional subspace $y\subseteq B_2$,
$y\not\subseteq A$ such that $A+y\in\mathcal{I}$. Because $B_1\cap B_2\subseteq A$ we have that $y\not\subseteq B_1$. We show that $A+y$ is in $\mathcal{B}$.
Suppose not, then there is an $B_3\in\mathcal{B}$ such that $A+y\subset B_3$ and $\dim(A+y)=\dim B_1<\dim B_3$, which contradicts (B2').
So (B3) holds. \\
Finally, for (B4) it is enough to notice that by (I3) every independent space is contained in a basis. So a maximal independent subspace of $A\subseteq E$ is the same as a maximal intersection between a member of $\mathcal{B}$ and $A$. Then (B4) is just a re-formulation of (I4) in terms of bases instead of independent spaces. \\
$\bullet$ Part 2. Let $\mathcal{B}$ be a family of subspaces of a finite dimensional space $E$ satisfying (B1),(B2),(B3),(B4).
Define $\mathcal{I}=\{I:\exists B\in\mathcal{B}, I\subseteq B \}$. We need to show $\mathcal{I}$ satisfies (I1),(I2),(I3),(I4). \\
Since $\mathcal{B}\neq\emptyset$ by (B1) and $\mathcal{B}\subseteq\mathcal{I}$, it follows that $\mathcal{I}\neq\emptyset$ and thus (I1). \\
To verify (I2), we need to show that if $I^\prime\subseteq I$ for some $I\in\mathcal{I}$, then $I^\prime\in\mathcal{I}$.
By the construction of $\mathcal{I}$, we know $I\subseteq B$ for some $B\in\mathcal{B}$. But then $I^\prime\subseteq I\subseteq B$ and so $I^\prime\in\mathcal{I}$ and (I2). \\
Now we prove (I3). Let $I_1,I_2\in\mathcal{I}$ with $\dim I_1<\dim I_2$. We may assume without loss of generality that $I_2$ is a basis, by truncating the matroid sufficiently many times by Theorem \ref{t-trunc}. Now $I_1$ is contained in a basis $B_1$ and $I_2=B_2$ is a basis.
There exists a codimension $1$ subspaces $A$ of $B_1$ that contains $I_1$, since $\dim I_1<\dim I_2$. Furthermore, we can choose $A$ such that $B_1\cap I_2\subseteq A$.
Hence by (B3) there is a one dimensional subspace $y$ of $I_2$ such that $A+y$ is a basis and $\dim A+y = \dim I_2$.
Now $y$ is not contained in $A$, since $\dim A = \dim I_2 -1$.
Therefore $I_1+y \subseteq A+y$ and $A+y$ is independent.
So $I_1+y$ independent and $y$ is not contained in $I_1$.\\
Finally, for (I4) we have the same reasoning as in part 1: (I4) is a re-formulation of (B4) in terms of independent spaces instead of bases. \\
$\bullet$ Part 3. Given a family of subspaces $\mathcal{I}$ satisfying (I1),(I2),(I3),(I4) create a family
$\mathcal{B}=\{B\in\mathcal{I}:\forall B^\prime\in\mathcal{I}, B\subseteq B^\prime\Rightarrow B^\prime=B\}$.
Then use $\mathcal{B}$ to create a (possibly new) family
$\mathcal{I}^\prime=\{I:\exists B\in\mathcal{B}, I\subseteq B \}$. We want to show that $\mathcal{I}^\prime=\mathcal{I}$. Let $I\in\mathcal{I}$, then $I\subseteq B$ for some $B\in\mathcal{B}$ that is of maximal dimension, so immediately $I\in\mathcal{I}^\prime$. On the other hand, if $I^\prime\in\mathcal{I}^\prime$, then $I^\prime\subseteq B$ for some $B\in\mathcal{I}$ of maximal dimension. By (I2), $I^\prime\in\mathcal{I}$, so $\mathcal{I}^\prime=\mathcal{I}$. \\
Given a family $\mathcal{B}$ satisfying (B1),(B2),(B3),(B4) create a family $\mathcal{I}=\{I:\exists B\in\mathcal{B}, I\subseteq B\}$. Then let $\mathcal{B}^\prime$ be the members of $\mathcal{I}$ of maximal dimension, that is,
$\mathcal{B}^\prime=\{B\in\mathcal{I}:\forall B^\prime\in\mathcal{I},B\subseteq B^\prime\Rightarrow B^\prime=B\}$. We will show that $\mathcal{B}^\prime=\mathcal{B}$. Let $B\in\mathcal{B}$ and suppose $B\notin\mathcal{B}^\prime$. Then $B$ is not a member of $\mathcal{I}$ of maximal dimension, so $B\subsetneq B^\prime$ for some $B^\prime\in\mathcal{I}$, which contradicts (B2'). Hence $B\in\mathcal{B}^\prime$. If $B^\prime\in\mathcal{B}^\prime$, then $B^\prime\in\mathcal{I}$, so $B^\prime\subseteq B$ for some $B\in\mathcal{B}$ Since $B^\prime$ is a member of $\mathcal{I}$ of maximal dimension and $B\in\mathcal{I}$ as well, we get $B^\prime\in\mathcal{B}$ so $\mathcal{B}^\prime=\mathcal{B}$.
\end{proof}

\section{Duality}

\begin{definition}
Let $M=(E,r)$ be a matroid and let
\[ r^*(A)=\dim A-r(M)+r(A^\perp) \]
be an integer-valued function defined on the subspaces of $E$. Then $M^*=(E,r^*)$ is the \emph{dual} of the $q$-matroid $M$.
\end{definition}
We need to show that this definition is well-defined, so that the the dual of a $q$-matroid is again a $q$-matroid.

\begin{theorem}
The dual $q$-matroid is indeed a $q$-matroid, that is, the function $r^*$ satisfies (r1),(r2),(r3).
\end{theorem}
\begin{proof}
Let $A,B\subseteq E$. We start with proving (r2), so assume $A\subseteq B$. Then $B^\perp\subseteq A^\perp$. This means we can find independent vectors $x_1,\ldots,x_k$ such that $A^\perp=B^\perp+x_1+\cdots+x_k$, where $k=\dim A^\perp-\dim B^\perp$. By repeating Lemma \ref{unit-rank-increase} multiple times, we find that
\[ r(A^\perp)\leq r(B^\perp)+k=r(B^\perp)+\dim A^\perp-\dim B^\perp. \]
We have the following equivalent statements:
\begin{eqnarray*}
r(A^\perp) & \leq & r(B^\perp)+\dim A^\perp-\dim B^\perp \\
r(A^\perp)-\dim A^\perp & \leq & r(B^\perp)-\dim B^\perp \\
r(A^\perp)+\dim E-\dim A^\perp & \leq & r(B^\perp)+\dim E-\dim B^\perp \\
r(A^\perp)+\dim A & \leq & r(B^\perp)+\dim B
\end{eqnarray*}
Then it follows that
\begin{eqnarray*}
r^*(A) & = & \dim A-r(M)+r(A^\perp) \\
 & \leq & \dim B-r(M)+r(B^\perp) \\
 & = & r^*(B)
 \end{eqnarray*}
and we have proved (r2). For (r1), notice that $r^*(\mathbf{0})=0-r(M)+r(E)=0$ and by (r2) it follows that $0\leq r^*(A)$ for all $A\subseteq E$. The other inequality of (r1) is proved via
\begin{eqnarray*}
r^*(A) & = & \dim A-r(M)+r(A^\perp) \\
 & \leq & \dim A-r(M)+r(M) \\
 & = & \dim A.
\end{eqnarray*}
We show (r3) using the modularity of dimension and semimodularity of $r$:
\begin{eqnarray*}
\lefteqn{r^*(A+B)+r^*(A\cap B)} \\
 & = & \dim(A+B)+\dim(A\cap B)-2\cdot r(E)+r((A+B)^\perp)+r((A\cap B)^\perp) \\
 & = & \dim A+\dim B-2\cdot r(E)+r(A^\perp\cap B^\perp)+r(A^\perp+B^\perp) \\
 & \leq & \dim A+\dim B-2\cdot r(E)+r(A^\perp)+r(B^\perp) \\
 & = & r^*(A)+r^*(B).
\end{eqnarray*}
Now $r^*$ satisfies (r1),(r2),(r3), so we conclude that the dual $q$-matroid is indeed a $q$-matroid.
\end{proof}

\begin{remark}\label{r-dual}
In the definition of duality we use the orthogonal complement of a subspace, with respect to the standard inner product in $E$. We could, however, have chosen any nondegenerate bilinear form on $E$ to define duality. Choosing another inner product (bilinear form) will result in isomorphic duals.
\end{remark}

An easy consequence of the definition of duality is the following:

\begin{corollary}
The rank of the dual $q$-matroid is $\dim E-r(M)$.
\end{corollary}
\begin{proof}
We have $r^*(M)=r^*(E)=\dim E-r(M)+r(\mathbf{0})=\dim E-r(M)$ as was to be shown.
\end{proof}

We can also characterize the bases of the dual $q$-matroid.

\begin{theorem}
Let $M=(E,r)$ be a $q$-matroid with $\mathcal{B}$ as collection of bases, and let
\[ \mathcal{B}^*=\{B^\perp:B\in\mathcal{B}\}. \]
Then the dual $q$-matroid $M^*$ has $\mathcal{B}^*$ as collection of bases.
\end{theorem}
\begin{proof}
The following statements are equivalent:
\begin{center}
\begin{tabular}{c}
$B$ is a basis of $M^*$ \\
$r^*(B)=\dim B=r^*(E)$ \\
$r(B^\perp)=r(M)$ and $\dim B=\dim E-r(M)$ \\
$r(B^\perp)=r(M)$ and $r(M)=\dim B^\perp$ \\
$B^\perp$ is a basis of $M$ \\
\end{tabular}
\end{center}
This proves that $M^*=(E,\mathcal{B}^*)$.
\end{proof}

This is a straightforward consequence of the theorem above:

\begin{corollary}
Let $M$ be a $q$-matroid. Then $(M^*)^*=M$.
\end{corollary}

\begin{example}
Consider the uniform $q$-matroid $U_{r,n}$ from Example \ref{ex-uniform}. We know that its bases are all subspaces of dimension $r$. This means the bases of the dual are all subspaces of dimension $n-r$. Thus, the dual of $U_{r,n}$ is $U_{n-r,n}$.
\end{example}

We have discussed in Section \ref{sec-codes} that rank metric codes give rise to $q$-matroids.
We show that the $q$-matroid associated to the dual code is the same as the dual of the $q$-matroid associated to the code.

\begin{theorem}\label{thm-dualmatroidC}
Let $K\subseteq L$ be a finite Galois field extension and let $C\subseteq L^n$ be a rank metric code.
Let $M(C)$ be the $q$-matroid associated to the code $C$. Then $M(C)^*=M(C^\perp)$.
\end{theorem}
\begin{proof}
We will show that both matroids have the same set of bases. Let $C$ be $k$ dimensional rank metric code over $L$ and let $G$ be a generator matrix of $C$.
A basis of $M(C)^*$ is of the form $B^\perp$ where $B$ is a basis of $M(C)$. Pick such a basis $B$ of $M(C)$, then $r(B)=r(M)=\dim B=k$.
After a $K$-linear coordinate change of $K^n$ we may assume without loss of generality that $B$ has generator matrix $Y=(I_k | O)$. (See Berger \cite{berger:2003} for more details on rank metric equivalence.) \\
Let $G=(G_1|G_2)$, where $G_1$ consists of the first $k$ columns of $G$ and $G_2$ consists of the last $n-k$ columns of $G$.
Then $GY^T=G_1$ is a generator matrix of $C_B$. Now $\dim_L(C_B)=k$, since $B$ is a basis. So $C_B=L^k$.
Hence, after a base change of $C$ we may assume without loss of generality that $C$ has generator matrix $G'=(I_k | P)$.
Therefore $H=(-P^T|I_{n-k})$ is a parity check matrix of $C$ and a generator matrix of $C^\perp$.
Now $Z=(O|I_{n-k})$ is a generator matrix of $B^\perp$ and $HZ^T=I_{n-k}$ is a generator matrix of $(C^\perp)_{B^\perp}$.
So $(C^\perp)_{B^\perp}=L^{n-k}$ and $B^\perp$ is a basis of $M(C^\perp)$. Therefore $\mathcal{B}(M(C)^*)\subseteq\mathcal{B}(M(C^\perp))$. \\
The other inclusion follows from using duality and replacing $C$ by $C^\perp$, leading to the following equivalent statements:
\begin{center}
\begin{tabular}{c}
$\mathcal{B}(M(C)^*)\subseteq\mathcal{B}(M(C^\perp))$ \\
$\mathcal{B}^*(M(C)^*)\subseteq\mathcal{B}^*(M(C^\perp))$ \\
$\mathcal{B}(M(C))\subseteq\mathcal{B}(M(C^\perp)^*)$ \\
$\mathcal{B}(M(C^\perp))\subseteq\mathcal{B}(M(C)^*)$ \\
\end{tabular}
\end{center}
We conclude that $\mathcal{B}(M(C)^*)=\mathcal{B}(M(C^\perp))$ and hence $M(C)^*=M(C^\perp)$.
\end{proof}

\begin{corollary}
The minimum rank distance $d_R(C)$ of a rank metric code $C$ is determined by $M(C)$. Moreover, rank metric codes that give rise to the same $q$-matroid will have the same minimum rank distance.
\end{corollary}
\begin{proof}
This is a direct consequence of Proposition \ref{pJ} and Theorem \ref{thm-dualmatroidC}.
\end{proof}

\begin{example}
Let $L=\mathbb{F}_{8} $ and $K=\mathbb{F}_2 $. Let $a\in\mathbb{F}_{8}$ with $a^3=1+a$. Let $C^\perp$ be the rank metric code that is the dual of the code defined in Example \ref{ex-code}. It is generated by
\[ H=\left( \begin{array}{cccc}
a^2 & a & 1 & 0 \\
0 & 0 & 0 &1
\end{array}\right). \]
We have seen that $M(C)$ is the $q$-matroid we defined in Example \ref{ex-2}. Its bases are the $2$-dimensional subspaces of $E=\mathbb{F}_2^4$ that do not contain $\langle0001\rangle$. This means the bases of $M(C)^*$ are the $2$-dimensional subspaces of $E$ that do not have $\langle0001\rangle$ in their complement. We check that these spaces are indeed the bases of $M(C^\perp)$. As argued in Example \ref{ex-code}, we need to show that there are no nonzero codewords of $C^\perp$ such that $\Rsupp(\mathbf{c})\subseteq B$, where $B$ is a basis of $M(C)$ (which is the orthogonal complement of a basis of $M(C)^*$). But $\Rsupp(\mathbf{c})$ for a nonzero word of $C^\perp$ has either dimension $3$, because $a^2$, $a$ and $1$ are algebraically independent in $\mathbb{F}_8$, or it is a multiple of $\langle0001\rangle$. In both cases we can not have that $\Rsupp(\mathbf{c})\subseteq B$. So we find that the bases of $M(C^\perp)$ are the same as the bases of $M(C)^*$, hence the two $q$-matroids are the same.
\end{example}

We conclude this section with a definition we will need later.

\begin{definition}
A $1$-dimensional subspace that is not in the orthogonal complement of a basis is called an \emph{isthmus}.
\end{definition}

\begin{corollary}
Let $e$ be a loop of the $q$-matroid $M$. Then $e$ is an isthmus of the dual $q$-matroid $M^*$.
\end{corollary}

\section{Restriction and contraction}

\begin{definition}
Let $M=(E,r)$ be a $q$-matroid and let $H$ be a hyperplane of $E$ that contains at least one basis of $M$. Then the \emph{restriction} $M|_H$ is a $q$-matroid with ground space $H$ and rank function
\[ r_{M|_H}(A)=r_M(A) \]
defined on the subspaces $A\subseteq H$.
\end{definition}

Before proving that restriction is well defined, a remark on deletion. For ordinary matroids, deletion of an element $e$ is the same as restriction to the complement of $e$. For $q$-matroids, we could say that restriction to $H$ is the same as deletion of the $1$-dimensional subspace $e$ orthogonal to $H$. However, since $H$ might contain $e$, the term ``deletion of $e$'' is a bit misleading. Therefore we prefer to talk about restriction.

\begin{theorem}
The restriction $M|_H$ is indeed a $q$-matroid, that is, $r_{M|_H}$ satisfies (r1),(r2),(r3).
\end{theorem}
\begin{proof}
For all $A\subseteq H$, we have that $A\subseteq E$. Hence the function $r_{M|_H}$ inherits the properties (r1),(r2),(r3) directly from $r_M$. We conclude that $M|_H$ is indeed a $q$-matroid.
\end{proof}

\begin{definition}
Let $M=(E,\mathcal{I})$ be a $q$-matroid and let $e$ be a $1$-dimension subspace of $E$ that is not a loop. Consider the projection $\pi:E\to E/e$. For every $A\subseteq E/e$, let $B$ be the unique subspace of $E$ such that $e\subseteq B$ and $\pi(B)=A$. Then the \emph{contraction} $M/e$ is a $q$-matroid with ground space $E/e$ and rank function
\[ r_{M/e}(A)=r_M(B)-1 \]
defined on the subspaces $A\subseteq E/e$.
\end{definition}

\begin{theorem}
The contraction $M/e$ is indeed a $q$-matroid, that is, $r_{M/e}$ satisfies (r1),(r2),(r3).
\end{theorem}
\begin{proof}
Note that $\dim B=\dim A+1$. Because $e\subseteq B$ and $e$ is not a loop, $r_M(B)\geq1$ hence $r_{M/e}(A)\geq0$. Since $r_M(B)\leq\dim B=\dim A+1$, we have $r_{M/e}(A)\leq\dim A$. This proves (r1). For (r2), let $A_1\subseteq A_2\subseteq E/e$ with corresponding $B_1,B_2\subseteq E$. Then $B_1\subseteq B_2$ so $r_M(B_1)\leq r_M(B_2)$, and it follows that $r_{M/e}(A_1)\leq r_{M/e}(A_2)$. \\
For (r3), take $A_1,A_2\subseteq E/e$ with corresponding $B_1,B_2\subseteq E$. Since $\pi$ preserves inclusion, we have that $\pi(B_1\cap B_2)=\pi(B_1)\cap\pi(B_2)=A_1\cap A_2$, and because $\pi$ is a homomorphism, we have that $\pi(B_1+B_2)=\pi(B_1)+\pi(B_2)=A_1+A_2$. Hence
\begin{eqnarray*}
r_{M/e}(A_1+A_2)+r_{M/e}(A_1\cap A_2) & = & r_M(B_1+B_2)-1+r_M(B_1\cap B_2) \\
 & \leq & r_M(B_1)-1+r_M(B_2)-1 \\
 & = & r_{M/e}(A_1)+r_{M/e}(A_2).
\end{eqnarray*}
This proves (r3). We conclude that $M/e$ is indeed a $q$-matroid.
\end{proof}

Before we give examples, we describe the independent spaces of restriction and contraction.

\begin{theorem}
Let $M=(E,r)$ be a $q$-matroid. Let $e$ be a $1$-dimension subspace of $E$ and consider the projection $\pi:E\to E/e$. Then the the independent spaces of the restriction to $e^\perp$ and the contraction of $e$ are given by
\begin{itemize}
\item Restriction: $\mathcal{I}(M|_{e^\perp})=\{I\in\mathcal{I}(M):I\subseteq e^\perp\}$
\item Contraction: $\mathcal{I}(M/e)=\{\pi(I):I\in\mathcal{I}(M),e\subseteq I\}$
\end{itemize}
\end{theorem}
\begin{proof}
For restriction this is quite clear: a subspace $I$ is independent in $M|_{e^\perp}$ if $r_{M|_{e^\perp}}(I)=\dim I$. By definition, this means $r_M(I)=\dim I$, so $I$ is independent in $M$.
For contraction, let $I$ be an independent subspace of $M/e$. Let $J$ be the unique subspace of $E$ such that $\pi(J)=I$ and $e\subseteq J$. Then we have that
\[ \dim I=r_{M/e}(I)=r_M(J)-1\leq \dim J-1=\dim I, \]
so equality must hold everywhere. Hence $r_M(J)=\dim J$ and $J$ is independent in $M$.
\end{proof}

\begin{example}\label{ex-uniform-delcon}
Let $U_{k,n}$ be the uniform $q$-matroid of Example \ref{ex-uniform} and let $e$ be a $1$-dimensional subspace of $E$. Then the restriction $U_{k,n}|_{e^\perp}$ has as independent spaces all subspaces of dimension at most $k$ that are contained in $e^\perp$. So $U_{k,n}|_{e^\perp}=U_{k,n-1}$ for any $e$. The contraction $U_{k,n}/e$ has as independent subspaces all subspaces of dimension at most $k$ containing $e$, mapped to $E/e$. This gives all subspaces in $E/e$ of dimension at most $k-1$. So $U_{k,n}/e=U_{k-1,n-1}$ for any $e$.
\end{example}

\begin{example}
Let $M$ be the matroid of Example \ref{ex-2} and let $e=\langle(0,0,0,1)\rangle$. Then we can not contract $e$, since it is a loop. But we can restrict to $e^\perp$. The independent spaces that are contained in $e^\perp$ can not contain $e$, because $e$ is not in $e^\perp$. This means all subspaces of $e^\perp$ of dimension $2$ or less are independent in the restriction, hence $M|_{e^\perp}$ is the uniform matroid $U_{2,3}$.
\end{example}

From now on, we will always assume that we never restrict to a hyperplane that does not contain a bases, nor contract loops. So if we talk about $M|_{e^\perp}$ we will assume $e$ is not an isthmus and if we talk about $M/e$ we assume $e$ is not a loop. The following observations are necessary to prove that restriction and contraction are dual operations:

\begin{remark}
Since $e^\perp$ and $E/e$ are both vector spaces over the same field of the same dimension $r(M)-1$, they are isomorphic. We construct an explicit isomorphism as follows. Recall that that all subspaces of $E/e$ can be obtained by $\pi(A)$ with $A\subseteq E$ and $e\subseteq A$. This gives an isomorphism between the subspaces of $E$ that contain $e$ and the subspaces of $E/e$. On the other hand, for a subspace $A$ that contains $e$ we can take the orthogonal complement of $e$ inside $A$ by restricting the inner product of $E$ to $A$. The result is in $e^\perp$.
\end{remark}

\begin{definition}
We denote bij $\varphi:E/e\to e^\perp$ the isomorphism taking $\pi(A)$ to the orthogonal complement of $e$ in $A$. On $e^\perp$ we have a canonical inner product, which is the restriction of the inner product of $E$. We denote it by $\langle\mathbf{x},\mathbf{y}\rangle_{e^\perp}$. Using $\varphi$, we can use it to define an inner product (bilinear form) $\langle\mathbf{x},\mathbf{y}\rangle_{E/e}$ on $E/e$ given by $\langle\mathbf{x},\mathbf{y}\rangle_{E/e}=\langle\varphi(\mathbf{x}),\varphi(\mathbf{y})\rangle_{e^\perp}$.
\end{definition}

\begin{theorem}
Let $M$ a $q$-matroid and $e\subseteq E$ not a loop or isthmus. Then restriction and contraction are dual notions, that is, $M^*/e=(M|_{e^\perp})^*$ and $(M/e)^*=M^*|_{e^\perp}$.
\end{theorem}
\begin{proof}
First, recall from Remark \ref{r-dual} that duality does not depend on the chosen inner product. We have the following equivalent statements:
\begin{center}
\begin{tabular}{c}
$\pi(B)\subseteq E/e$ is a basis of $M^*/e$ \\
$B\subseteq E$ is a basis of $M^*$ and $e\subseteq B$ \\
$B^\perp\subseteq E$ is a basis of $M$ and $B^\perp\subseteq e^\perp$ \\
$B^\perp\subseteq e^\perp$ is a basis of $M|_{e^\perp}$ \\
\end{tabular}
\end{center}
On the other hand, we have that $\varphi(\pi(B))\subseteq e^\perp$ and this is the orthogonal complement of $B^\perp$ in $e^\perp$. This shows that a basis in $M^*/e$ is isomorphic to a basis in $(M|_{e^\perp})^*$ and hence $M^*/e=(M|_{e^\perp})^*$. For the other equality, use duality and replace $M$ by $M^*$ to get the following equivalent statements:
\begin{eqnarray*}
M^*/e & = & (M|_{e^\perp})^* \\
(M^*/e)^* & = & M|_{e^\perp} \\
(M/e)^* & = & M^*|_{e^\perp}
\end{eqnarray*}
This proves that restriction and contraction are dual operations.
\end{proof}

\section{Towards more cryptomorphisms}\label{sec-morecrypt}

An important strength of ordinary matroids is that they have so may cryptomorphic definitions. For $q$-matroids we already saw a definition in terms of the rank function, independent spaces, and bases. We saw that taking the $q$-analogue of two cryptomorphic definitions of a matroid can result in statements that are not cryptomorphic. In this section we lay some ground work for more cryptomorphisms.

\subsection{Circuits}

\begin{definition}
Let $M=(E,\mathcal{I})$ be a $q$-matroid and let $C\subseteq E$. Then $C$ is a \emph{circuit} of $M$
if $C$ is a dependent subspace of $E$ and every proper subspace of $C$ is independent.
\end{definition}

\begin{example}
Let $U_{k,n}$ be the uniform $q$-matroid of Example \ref{ex-uniform}. Its circuits are all the subspaces of $E$ of dimension $k+1$.
\end{example}

\begin{example}
Let $M$ be the $q$-matroid of Example \ref{ex-2}. Its circuits are the $3$-dimensional spaces not containing $\langle0001\rangle$ and the $2$ dimensional spaces that do contain $\langle0001\rangle$.
\end{example}

The circuits of a $q$-matroid satisfy the following properties.

\begin{theorem}
Let $M=(E,\mathcal{I})$ be a $q$-matroid and $\mathcal{C}$ its family of circuits. Then $\mathcal{C}$ satisfies:
\begin{itemize}
\item[(C1)] $\mathbf{0}\notin\mathcal{C}$
\item[(C2)] If $C_1,C_2\in\mathcal{C}$ and $C_1\subseteq C_2$, then $C_1=C_2$.
\item[(C3)] If $C_1,C_2\in\mathcal{C}$ distinct and $x\subseteq C_1\cap C_2$ a $1$-dimensional subspace, then there is a $C_3\subseteq C_1+C_2$ with $x\not\subseteq C_3$ so that $C_3\in\mathcal{C}$.
\end{itemize}
\end{theorem}
\begin{proof}
Since $\mathbf{0}$ is independent by (I1'), it is not a circuit and thus (C1) holds. (C2) follows from the definition of a circuit. \\
To show (C3), let $C_1,C_2\in\mathcal{C}$ with nontrivial intersection. The space $C_1+C_2$ is dependent, since it contains $C_1$ and $C_2$, so it has to contain at least one circuit. We have to prove that for a $1$-dimensional $x\subseteq C_1\cap C_2$ we have such a circuit that trivially intersects $x$. Consider a codimension $1$ subspace $D$ of $C_1+C_2$ that does not contain $x$. Then $\dim D=\dim(C_1+C_2)-1$. Assume that $D$ is independent. \\
Now we know that $C_1-C_2$ can not be empty, because then $C_1\subseteq C_2$ which violates (C2). Similarly, $C_2-C_1$ is nonempty. Let $X\subseteq C_1$ of codimension $1$ with $C_1\cap C_2\subseteq X$. Such an $X$ exists because $C_1-C_2$ is nonempty. $X$ is independent, because it is a proper subspace of a circuit. Use (I3) multiple times to extend $X$ to a maximal independent space in $C_1+C_2$, call it $Y$. Now $Y$ contains $C_1\cap C_2$, but it does not contain all of $C_1$ or $C_2$ by construction. So $\dim Y\leq(\dim C_1-1)+(\dim C_2-1)-\dim(C_1\cap C_2)\leq\dim(C_1+C_2)-2$. \\
We now have two independent spaces in $C_1+C_2$: $D$ and $Y$. But $\dim Y<\dim D$ contradicts the maximality of $Y$. So $D$ has to be dependent and we can find a circuit $C_3\subseteq D$ with $x\not\subseteq C_3$. This proves (C3).
\end{proof}

We can already say that these three properties (C1),(C2),(C3) will not be enough to determine a $q$-matroid, for the same reasons as mentioned in Remark \ref{r-axioms}. If we take the family of circuits of a $q$-matroid and embed them in a space of higher dimension, then the properties (C1),(C2),(C3) still hold, but Lemma \ref{loopsum} fails.

\subsection{Closure}

\begin{definition}
Let $M=(E,r)$ be a $q$-matroid. For all subspaces $A\subseteq E$ we define the \emph{closure} of $A$ as
\[ \cl(A)=\bigcup\{x\subseteq E: r(A+x)=r(A)\}. \]
So $\cl$ is a function from the subspaces of $E$ to the subspaces of $E$. If a subspace is equal to its closure, we call it a \emph{flat}.
\end{definition}

Note that the closure is in fact a subspace, by Proposition \ref{p-rank2} and (r3).

\begin{example}
Let $U_{r,n}$ be the uniform $q$-matroid of Example \ref{ex-uniform}. All subspaces of dimension at most $k-1$ are flats, since adding a $1$-dimensional subspace will increase the rank. The closure of a basis is the whole space $E$ -- in fact, this is true for any $q$-matroid.
\end{example}

\begin{example}
Let $M$ be the $q$-matroid of Example \ref{ex-2}. To find the closure of a $1$-dimensional space, we can always add the loop $\langle0001\rangle$.
\end{example}

The closure satisfies the following properties.

\begin{theorem}
Let $M=(E,r)$ be a $q$-matroid and $\cl$ its closure. Then $\cl$ satisfies for all $A,B\subseteq E$ and $1$-dimensional subspaces $x,y\subseteq E$:
\begin{itemize}
\item[(cl1)] $A\subseteq\cl(A)$
\item[(cl2)] If $A\subseteq B$ then $\cl(A)\subseteq\cl(B)$.
\item[(cl3)] $\cl(A)=\cl(\cl(A))$
\item[(cl4)] If $y\subseteq\cl(A+x)$ and $y\not\subseteq\cl(A)$, then $x\subseteq\cl(A+y)$.
\end{itemize}
\end{theorem}
\begin{proof}
Property (cl1) follows directly from the definition of closure. For (cl2), assume $A\subseteq B$. By (r3) we have that
\[ r(\cl(A)+B)+r(\cl(A)\cap B)\leq r(\cl(A))+r(B)=r(A)+r(B). \]
Because $A\subseteq\cl(A)\cap B$ we have by (r2) that $r(A)\leq r(\cl(A)\cap B)$. Combing gives that $r(\cl(A)+B)\leq r(B)$. On the other hand, $B\subseteq\cl(A)+B$, hence (r2) gives that $r(B)\leq r(\cl(A)+B)$. It follows that equality must hold, so $r(B)=r(\cl(A)+B)$ and therefore $B+\cl(A)\subseteq\cl(B)$. Finally, since $\cl(A)\subseteq\cl(A)+B$, it follows that $\cl(A)\subseteq\cl(B)$. \\
We prove (cl3) by proving the two inclusions. From (cl1) it follows that $\cl(A)\subseteq\cl(\cl(A))$. For the other inclusion, let $x\subseteq\cl(\cl(A))$ be a $1$-dimensional subspace. Then we have $r(\cl(A)+x)=r(\cl(A))+r(A)$. But by (r2), we have $r(\cl(A)+x)\geq r(A+x)\geq r(A)$, so equality must hold in throughout this statement. It follows that $x\subseteq\cl(A)$, hence $\cl(\cl(A)\subseteq\cl(A)$ and (cl3) is proved. \\
To prove (cl4), let $y\subseteq\cl(A+x)$ and $y\not\subseteq\cl(A)$. Then $r(A+x+y)=r(A+x)$ and $r(A+y)\neq r(A)$, so by Lemma \ref{unit-rank-increase} it follows that $r(A+y)=r(A)+1$. We have that
\[ r(A)+1=r(A+y)\leq r(A+y+x)=r(A+x)\leq r(A)+1, \]
so equality must hold everywhere. This means $r(A+y)=r(A+y+x)$, hence $x\subseteq\cl(A+y)$ and we have proved (cl4).
\end{proof}

It is not known if these properties (cl1),(cl2),(cl3),(cl4) are enough to completely determine a $q$-matroid.

\section{Further research directions}

We have established the definitions and several basic properties of $q$-matroids. However, this is just the beginning of the research: in potential, all that is known about matroids could have a $q$-analogue. In this section we make a modest (and somewhat personal) wish-list on where to go next with the research in $q$-matroids. \\

In a late stadium, we learned about the work of Crapo \cite{crapo:1964} on a very closely related topic. Defining an ordinary matroid by its rank function can be viewed as assigning a rank to every element of a Boolean lattice, in such a way that the following properties hold:
\begin{itemize}
\item[(r1)] $0\leq r(A)\leq h(A)$
\item[(r2)] If $A\leq B$, then $r(A)\leq r(B)$.
\item[(r3)] $r(A\vee B)+r(A\wedge B)\leq r(A)+r(B)$
\end{itemize}
Here $h(A)$ is the hight of $A$ in the Boolean lattice, that is, the size of the subset. Join and meet in the Boolean lattice correspond to union and intersection. In this work, we assign a rank with the same properties to every element in a (finite) subspace lattice. The hight of an element in the subspace lattice is its dimension, and the equivalents of join and meet are sum and intersection. The work of Crapo generalises this idea: it turns out that for every complemented modular lattice, one can give the elements a rank function that satisfies the above properties. \\
So, this work on $q$-matroids can be viewed as a special case of the work of Crapo. Where Crapo's motivation and point of view are much more combinatorial, our work relies heavily on linear algebra and therefore might not be easily generalised. We strongly believe that a combination of the two approaches can greatly benefit the study of $q$-matroids. \\

There are many more ways to define matroids that probably have a $q$-analogue. For example in terms of circuits, flats, hyperplanes, or the closure function. First steps in this direction were taken in Section \ref{sec-morecrypt}. Another property of matroids that could have a $q$-analogue is that of connectivity and the direct sum. Special properties of matroids for which we want to decide there is a $q$-analogue include Pappus, Desargues and Vamos. \\

The motivation to study $q$-matroids comes from rank metric codes. There is a link between the weight enumerator of a linear code (in the Hamming metric) and the Tutte polynomial of the associated matroid. It can be established via the function $l(J)$. Can we do the same for $q$-matroids and rank metric codes? \\
To answer this question, we must first find the right definition of the Tutte polynomial. Originally, it was defined in terms of internal and external activity of bases of a matroid. It seems not so easy to do the same for $q$-matroids. A better place to start would be the rank generating polynomial:
\[ R_M(X,Y)=\sum_{A\subseteq E}X^{r(E)-r(A)}Y^{|A|-r(A)}. \]
First notice that in order to get a finite sum, we need $E$ to be a vector space over a finite field -- or maybe we need a different definition to begin with. In the case of a finite field the formula above has a straightforward $q$-analogue: just replace $|A|$ with $\dim A$. For normal matroids, this polynomial is equivalent to the Tutte polynomial. Greene \cite{greene:1976} was the first to prove the link between the Tutte polynomial and the weight enumerator. He used that both behave the same under deletion and contraction. How would that work in $q$-matroids? This is by no means straightforward. In ordinary matroids, we have that $|\mathcal{B}(M-e)|+|\mathcal{B}(M/e)|=|\mathcal{B}(M)|$, which can be used to show the relation between the Tutte polynomials (hence rank generating polynomials) of $M$, $M-e$ and $M/e$. For $q$-matroids, life is less pretty. $\mathcal{B}(M|_{e^\perp})$ comes from the bases of $M$ that are contained in $e^\perp$ while $\mathcal{B}(M/e)$ comes from the bases of $M$ that contain $e$. Because of self-duality in finite vector spaces, these families are not disjoint and also together they do not have to give all bases of $M$. \\
Another question regarding the Tutte polynomial, that looks easier to solve, is how it behaves under duality. \\

For linear error-correcting codes and matroids, the notions of puncturing and shortening of codes generalize to deletion and contraction in matroids. For rank metric codes, the operations of puncturing and shortening are studied in \cite{martinez-penas:2016}. Linking the notions of restriction and contraction of $q$-matroids and puncturing and shortening in rank metric codes should help to find a $q$-analogue for the proof of Greene \cite{greene:1976} of the link between the Tutte polynomial and the weight enumerator. \\

We can consider $q$-matroids that arise from rank metric codes as \emph{representable}, analogous to the case for normal matroids. Are all $q$-matroids representable? A big difference between normal matroids and $q$-matroids is that all uniform $q$-matroids are representable by MRD codes, as we have seen in Example \ref{ex-uniform-MRD}, and MRD codes are known tho exist for all parameters over finite fields \cite{gabidulin:1985}, in characteristic zero \cite{augot:2013}, as well as over rational function fields \cite{augot:2014}. \\

A very important reason why matroids are studies extensively, is that they are generalizations of many objects in discrete mathematics. It is interesting to see if this holds for $q$-matroids as well. It is known \cite{deza:1992} that Steiner systems give matroids, so called \emph{perfect matroid designs}: these are matroids where all flats of the same rank have the same size. Do $q$-ary Steiner systems, the $q$-analogue of Steiner systems, also give us a special kind of $q$-matroids? Currently, there is only one $q$-ary Steiner system known \cite{braun:2013}. Perfect matroid designs have been used to construct new Steiner systems. If a $q$-analogue of a perfect matroid design exists, it provides a new tool in the search for $q$-ary Steiner systems. \\

Matroids generalize graphs and graphs are an important class of matroids. For $q$-matroids, it is not clear if they generalize a $q$-analogue of a graph. We would expect that if such analogy exists, it follows directly from the notion of circuits of $q$-matroids. There are some results about $q$-Kneser graphs, see for example \cite{mussche:2009}, which are the $q$-analogues of Kneser graphs. But these $q$-Kneser graphs are still ``ordinary'' graphs, so it is unlikely that they play the role to $q$-matroids as graphs do for matroids. \\

To summarize, we think that one should study $q$-matroids for the same reasons one should study matroids. There are a lot of problems and questions regarding $q$-matroids waiting for interested researchers.

\section*{Acknowledgement}

This paper has been ``work in progress'' for quite some time. The authors would like to thank all colleagues who have discussed the subject with us over time. Specifically, the members of the COST action ``Random network coding and designs over GF(q)'' and the attendants of the 2016 International Workshop on Structure in Graphs and Matroids. We are grateful to Henry Crapo for sending us his thesis \cite{crapo:1964} and sharing his approach on the subject.

\bibliographystyle{plain}
\bibliography{qmatroid}

\end{document}